\pdfoutput=1 
\documentclass[11pt, oneside, dvipsnames, svgnames, table]{amsart}
\usepackage[letterpaper,margin=2.3cm]{geometry}
\usepackage{sgheader}
\usepackage[T1]{fontenc}

\begin{document}

\sloppy

\title{A geometric splitting of the motive of $\GL_n$}
\author{W. Sebastian Gant}
\address{W. Sebastian Gant. Department of Mathematics, University of British Columbia, Vancouver~BC V6T~1Z2, Canada.}
\email{wsgant@math.ubc.ca}

\begin{abstract} 
    A paper by Haynes Miller shows that there is a filtration on the unitary groups that splits in the stable homotopy category, where the stable summands are certain Thom spaces over Grassmannians. We give an algebraic version of this result in the context of Voevodsky's tensor triangulated category of stable motivic complexes $\cat{DM}(k,R)$, where $k$ is a field. Specifically, we show that there are algebraic analogs of the Thom spaces appearing in Miller's splitting that give rise to an analogous splitting of the motive $M(\GL_n)$ in $\cat{DM}(k,R)$, where $\GL_n$ is the general linear group scheme over $k$. 
\end{abstract}

\maketitle

\section{Introduction}\label{sec:Intro}

It is shown in \cite{Mil:StableSplittings} that there is a filtration on the unitary group $\U(n)$:
\[
    \varnothing = F_{-1}(n) \subseteq F_0(n)  \subseteq F_1(n) \subseteq \cdots \subseteq F_{n-1}(n) \subseteq F_n(n) = \U(n),
\]
where 
\[
    F_m(n) := \{A \in \U(n) : \dim \ker(A - I_n) \geq n-m\},
\]
that splits in the stable homotopy category. That is, there is a filtration preserving stable weak equivalence
\begin{equation}\label{eq:MilSplitting}
    \U(n)_+ \xrar{\sim} \bigvee_{m=0}^n F_m(n) / F_{m-1}(n),
\end{equation}
where the target is filtered by truncating wedge summands. The quotient $F_m(n) / F_{m-1}(n)$ is identified \textit{op. cit.} as the Thom space of the rank-$m^2$ vector bundle over $\Gr(m,n)(\C)$ associated with the adjoint representation of $\U(m)$ on its Lie algebra. 

Recall that the singular cohomology ring of $\U(n)$ admits a presentation 
\[
    \h^*(\U(n),\Z) = \Lambda_\Z(\rho_1^u, \dots, \rho_n^u)
\]
where $|\rho_i^u| = 2i-1$. The cohomology ring $\h^*(\U(n),\Z)$ has a second grading given by word length, and we denote by
\[
    \Lambda_\Z^m(\rho_1^u, \dots, \rho_n^u)
\]
the $m$-th graded piece for this grading. The stable splitting \eqref{eq:MilSplitting} and the Thom isomorphism theorem provide a decomposition
\begin{equation}\label{eq:UCohoThomDecomp}
    \h^*(\U(n),\Z) \cong \bigoplus_{m=0}^n \h^{*+m^2}(\Gr(m,n)(\C),\Z)
\end{equation}
as a graded abelian group. It follows readily from the work of \cite{Crabb:StableSummands} (see \Cref{lem:Lev}) that the summand $\h^{*+m^2}(\Gr(m,n)(\C),\Z)$ corresponds to the free submodule $\Lambda_\Z^m(\rho_1^u, \dots, \rho_n^u)$ of $\h^*(\U(n),\Z)$. 

The quotients $F_m(n)/F_{m-1}(n)$ admit another description. Let $V(m,n)^u$ denote the complex Stiefel manifold of orthonormal $m$-frames in $\C^n$, equipped with its usual action by $\U(m)$,  and let $\U(m)^{ad}$ denote the Lie group $\U(m)$ equipped with the adjoint action. Define
\[
    A(m,n)^u := V(m,n)^u \times_{\U(m)} \U(m)^{ad}.
\]
Points of $A(m,n)^u$ are pairs $(V,\phi)$ consisting of an $m$-plane $V \subseteq \C^n$ and a unitary automorphism $\phi$ of $V$, and projection to the first factor gives $A(m,n)^u$ the structure of a bundle of groups over the complex Grassmannian $\Gr(m,n)(\C)$ with typical fibre $\U(m)$. There is a map $f_{m,n}\co A(m,n)^u \to \U(n)$, defined by sending $(V,\phi)$ to the unitary transformation 
\[
    \phi \oplus \id \co V \oplus V^\perp \to  V \oplus V^\perp
\]
of $\C^n$, which surjects onto $F_m(n)$. The filtration $\{F_l(m)\}$ on $\U(m)^{ad}$ is equivariant so induces a filtration on $A(m,n)^u$ which we denote by $\{F_l(m,n)\}$. Miller shows that the map $f_{m,n}$ is filtration-preserving and induces a homeomorphism 
\[
    A(m,n)^u/F_{m-1}(m,n) \to F_m(n)/F_{m-1}(n)
\]
on quotients.

A motivic analog of the map $f_{m,n}$ has been studied by Ben Williams in \cite{Wil:MotivicCoho} for the case $m=1$. Note that $A(1,n)^u$ is homeomorphic to $S^1 \times \CP^{n-1}$, and the map $f_{1,n}\co S^1 \times \CP^{n-1} \to \U(n)$ factors through $A(1,n)^u /F_0(1,n) = \Sigma \CP^{n-1}_+$. The motivic analog of $A(1,n)^u$ is the scheme $\G_m \times \tilde{\p}^{n-1}$, where $\tilde{\p}^{n-1}$ is the so-called Jouanolou's device for $\p^{n-1}$ (denoted $F^{n-1}$ in \cite{Wil:MotivicCoho}). Williams constructs a map of schemes
\[
    \G_m \times \tilde{\p}^{n-1} \to \GL_n
\]
which factors through the motivic space $\G_m \wedge (\tilde{\p}^{n-1}_+)$, where $\G_m$ is pointed at 1, and shows that the induced map in motivic cohomology 
\[
    \h^{*,*}(\GL_n,\Z) \to \h^{*,*}(\G_m \wedge (\tilde{\p}^{n-1}_+), \Z)
\]
is an isomorphism in bidegrees $(2i-1,i)$ for each $i$ (\cite{Wil:MotivicCoho}*{Thm. 18}). The motivic cohomology ring $\h^{*,*}(\GL_n,\Z)$ is generated as an algebra over $\h^{*,*}(k,\Z)$ by classes $\rho_i$ with $|\rho_i| = (2i-1,i)$, so Williams' calculation can be seen as a motivic analog of the fact that the summand $\h^{*+1}(\CP^{n-1},\Z)$ of \eqref{eq:UCohoThomDecomp} corresponds to the free submodule generated by the classes $\rho_i^u$. 

This paper builds on the work of \cite{Wil:MotivicCoho} to show that, under some conditions on the base field $k$ and coefficient ring $R$, the motive of $\GL_n$ admits a decomposition analogous to \eqref{eq:MilSplitting} in the category $\cat{DM}(k,R)$, Voevodsky's category of stable motivic complexes with coefficients in $R$ (see \cite{CD:TriagulatedCatsOfMotives}*{Def. 11.1.1} for a precise definition of this category). To this end, we construct algebraic analogs of the spaces $A(m,n)^u$, which we denote by $A(m,n)$, as well as the quotients $F_m(n)/F_{m-1}(n)$, whose algebraic versions we denote by $X_m(m,n)$. We show that there are algebraic versions of the maps $f_{m,n}$:
\[
    A(m,n) \to \GL_n,
\]
certain quotients of which split in $\cat{DM}(k,R)$ to yield a decomposition 
\[
    M(\GL_n) \cong \bigoplus_{m=0}^n M(X_m(m,n)).
\]
In the case that the base field $k$ admits a complex embedding, the motivic space $X_m(m,n)$ complex realizes to the quotient $F_m(n)/F_{m-1}(n)$ of Miller's filtration (\Cref{prop:XmRealization}). We conjecture that such a splitting exists in the stable motivic homotopy category $\cat{SH}(k)$.

\subsection{Outline}

In \Cref{sec:Prelims}, we show that the motive $M(\GL_n)$ is pure Tate (though this result is known, see e.g. \cite{Pen:MWMotivicCoho}) and discuss a key computational tool: the motivic Rothenberg--Steenrod spectral sequence. The motivic spaces $X_m(m,n)$ are somewhat challenging to construct and require a discussion of automorphism bundles over (Jouanolou's devices of) flag varieties. This discussion is carried out in \Cref{sec:AutBundles}. \Cref{sec:MotiveOfA} establishes that the motives associated with these bundles over flag varieties are pure Tate and identifies their Tate summands. Calculations with the motivic Rothenberg--Steenrod spectral sequence are carried out in \Cref{sec:SSCalcs}. The motivic spaces $X_m(m,n)$ are defined in \Cref{sec:AuxSpaces}, the relevant aspects of complex realization are discussed in \Cref{sec:CxReal}, and Sections \ref{sec:Char0Splitting} and \ref{sec:CharPSplitting} establish the splitting over fields of characteristic 0 and positive characteristic, respectively. 

\subsection{Notation and conventions}

Throughout this paper, $k$ denotes a base field and $R$ a (commutative, unital) coefficient ring. Both $k$ and $R$ may come with additional hypotheses which will be made clear at the beginning of each section. If a singular or motivic cohomology group appears without coefficients, then integer coefficients are assumed. Let $\cat{Sm}_k$ denote the category of smooth, separated, finite type $k$-schemes and $Sh_{\text{Nis}}(\cat{Sm}_k)$ the category of Ninsevich sheaves on $\cat{Sm}_k$. Objects of $\cat{Sm}_k$ will simply be called ``smooth $k$-schemes." Since we make use of complex realization, the motivic homotopy-theoretic constructions in this paper are carried out in the universal model category developed in \cite{DUG:UniversalHmtpy}*{Sec. 8}, which we denote by $\cat{Spc}(k)$. This model category is Quillen equivalent to the one constructed in \cite{MV:A1Homotopy} by \cite{DUG:UniversalHmtpy}*{Prop. 8.1}. If the base field is equipped with an embedding $k \to \C$, we denote the analytic space associated with an object $Y$ of $\cat{Sm}_k$ by $Y(\C)$. Lastly, we adopt the motivic cohomology grading convention for motivic spheres: we denote by $S^{p,q}$ the motivic sphere $(S^1)^{\wedge p-q} \wedge \G_m^{\wedge q}$, and when $Y$ is a pointed motivic space,  $\Sigma^{p,q} Y := S^{p,q} \wedge Y$.

\section{Preliminaries}\label{sec:Prelims}

In this section, we establish preliminary results regarding the motive $M(\GL_n)$ and describe the motivic Rothenberg--Steenrod spectral sequence as developed in \cite{Wil:EquivariantMotivicCoho}. 

It is shown in \cite{Ron:EndomorphismsofP2}*{Prop. 4.2} that the inclusion 
\begin{align*}
    \SL_{n-1} &\to \SL_n, \\
    g &\mapsto 
    \begin{bmatrix}
            g & 0 \\
            0 & 1
    \end{bmatrix}
\end{align*}
fits into in an $\A^1$-homotopy cofibre sequence 
\[
    (\SL_{n-1})_+ \to (\SL_n)_+ \to \Sigma^{2n-1,n} ({\SL_{n-1}}_+).
\]
After applying $- \wedge (\G_m)_+$ to the previous cofibre sequence, we obtain the following. 

\begin{lem}\label{lem:in-1Hocofib}
    The inclusion 
    \begin{align*}
        \imath_{n-1} \co \GL_{n-1} \to \GL_n \\
        g \mapsto 
        \begin{bmatrix}
            g & 0 \\
            0 & 1
        \end{bmatrix}
    \end{align*}
    fits into an $\A^1$-homotopy cofibre sequence 
    \[
        (\GL_{n-1})_+ \xrar{(\imath_{n-1})_+} (\GL_n)_+ \to \Sigma^{2n-1,n}({\GL_{n-1}}_+).
    \]
\end{lem}

\begin{notation}
    Following the notation of \cite{Pen:MWMotivicCoho}, given a strictly increasing sequence of positive integers $(i_1, \dots, i_m)$, we let $d(i_1, \dots, i_m)$ denote the bidegree $(\sum_{j=1}^m (2i_j-1),\sum_{j=1}^m i_j)$. Similarly, let $R(d(i_1, \dots, i_m))$ denote the Tate motive $R(\sum_{j=1}^m i_j)[\sum_{j=1}^m (2i_j-1)]$. If $m=0$, we set $R(d(i_1, \dots, i_m)) := R(0)[0]$.
\end{notation}

Recall the following definition.

\begin{defn}
    An object of $\cat{DM}(k,R)$ is \textit{pure Tate} if it is isomorphic to a finite direct sum of Tate motives $R(q)[p]$. 
\end{defn}

A straightforward induction argument (see \cite{Pen:MWMotivicCoho}*{Rem. 3}) in conjunction with \Cref{lem:in-1Hocofib} shows the following. 

\begin{prop}\label{prop:GLnPureTate}
    The motive $M(\GL_n)$ is pure Tate. In particular, there is an isomorphism 
    \[
        M(\GL_n) \cong \smashoperator[lr]{\bigoplus_{\substack{1 \leq i_1 < \cdots < i_m \leq n \\ 0 \leq m \leq n}}} R(d(i_{1},\dots, i_m))
    \]
    in $\cat{DM}(k,R)$ such that the inclusion $\imath_{n-1}\co \GL_{n-1} \to \GL_n$ induces the obvious inclusion of summands. 
\end{prop}

\subsection{The motivic Rothenberg--Steenrod spectral sequence}

In this subsection, we assume that the pair $(k,R)$ satisfies the Beilinson--Soul\'{e} vanishing conjecture: the motivic cohomology groups $\h^{p,q}(\Spec k,R)$ are trivial when $p < 0$. Known examples are when $k$ is a number field or a finite field and $R$ is any coefficient ring, or when $k$ contains an algebraically closed field and $R = \Z/\ell$, with $\ell$ prime to the characteristic of $k$ (\cite{A:BSZ/l}). It is unknown if the vanishing conjecture holds for general pairs $(k,R)$. 

First, we recall the following definition. 
\begin{defn}\label{def:PrincBund}
    Let $G$ be a group object in $Sh_{\text{Nis}}(\cat{Sm}_k)$. Suppose $Y$ is a Nisnevich $G$-sheaf with action map $a\co G \times Y \to Y$ and $Z$ is a Nisnevich sheaf on which $G$ acts trivially. A map of Nisnevich sheaves $\pi\co Y \to Z$ is a \textit{principal $G$-bundle} if $\pi$ is a $G$-equivariant epimorphism (of sheaves) and the map
    \[
		a \times \pr_2\co  G \times Y \to Y \times_Z Y
    \]
    is an isomorphism, where $\pr_2$ is projection onto the second factor. We say $\pi\co Y \to Z$ is \textit{Nisnevich (resp. Zariski) locally trivial} if there is a Nisnevich (resp. Zariski) cover $\{f_i\co U_i \to Z\}$ and sections $s_i\co U_i \to Y$ such that $\pi s_i = f_i$.
\end{defn}

The motivic Rothenberg--Steenrod spectral sequence computes the equivariant motivic cohomology of a scheme $Y$ with respect to a free group action by a group scheme $G$. The spectral sequence is developed in \cite{Wil:EquivariantMotivicCoho} for more general bigraded cohomology theories and those groups schemes which are finite cell complexes in the sense of \cite{DI:MotivicCellStructures}. This latter condition guarantees a K\"{u}nneth isomorphism for $G \times Y$, an important ingredient in determining the $\E_2$-page.

The classical version of this spectral sequence originally developed in \cite{RS:CohoOfClassifyingSpaces} and sometimes goes under the name ``Eilenberg-Moore spectral sequence" in the literature. However, it is not the usual Eilenberg-Moore spectral sequence converging to the cohomology of a pullback. 

Given a coefficient ring $R$, the inputs for the motivic Rothenberg--Steenrod spectral sequence are the motivic cohomology rings $\h^{*,*}(G,R)$ and $\h^{*,*}(Y,R)$ as well as the module structure of the dual of $\h^{*,*}(Y,R)$ as a module over the dual of $\h^{*,*}(G,R)$; this structure being induced by the action map $G \times Y \to Y$.  

\begin{notation}
    Let $\m_R$ denote the motivic cohomology ring $\h^{*,*}(\Spec k,R)$. For a bigraded module $N^{*,*}$ over the bigraded ring $\m_R$, we let $\widehat{N}$ denote the dual bigraded $\m_R$-module.
\end{notation}

The $\m_R$-algebra $\h^{*,*}(G,R)$ is a Hopf algebra, and in the case that both $\h^{*,*}(G,R)$ and $\h^{*,*}(Y,R)$ are finitely generated $\m_R$-modules (as in all the cases considered in this paper), the action map 
\[
	a\co  G \times Y \to Y
\]
induces a module structure on $\widehat{\h^{*,*}(Y,R)}$ over the ring $\widehat{\h^{*,*}(G,R)}$. The following proposition is a combination of \cite{Wil:EquivariantMotivicCoho}*{Thm. 1.15} and \cite{Wil:EquivariantMotivicCoho}*{Thm. 2.4}. 
\begin{prop}\label{prop:RSSS}
    Suppose the Beilinson-Soul\'{e} vanishing conjecture holds for the pair $(k,R)$. Let $G$ be a smooth group scheme over $k$ which is also a finite cell complex, and suppose $G$ acts on a smooth $k$-scheme $Y$ and acts trivially on a smooth $k$-scheme $Z$ such that $Y \to Z$ is a Nisnevich-locally trivial principal $G$-bundle. Suppose further that $\h^{*,*}(G,R)$ and $\h^{*,*}(Y,R)$ are finitely generated, free $\m_R$-modules generated by elements in the first quadrant. Then there is a trigraded strongly convergent spectral sequence of $\m_R$-algebras:
    \[
		\E_2^{l,p,q} = \Ext^l_{\widehat{\h^{*,*}(G,R)}}(\widehat{\h^{p,q}(Y,R)}, \m_R) \Rightarrow \h^{l+p,q}(Z,R)
    \]
    functorial in $G$, $Y$, and $R$ with differentials $d_s\co \E_s^{l,p,q} \to \E_s^{l+s,p-s+1,q}$. 
\end{prop}

\section{Automorphism bundles over flag varieties}\label{sec:AutBundles}

In this section, we introduce the $k$-schemes of central interest, denoted $A(n_1, \dots, n_r)$. We make no assumptions on $k$ or $R$. In fact, this section can be carried out over $\Spec \Z$. 

Given a sequence of nonnegative integers $n_1 < n_2 < \dots < n_r$, let $\Fl(n_1, \dots, n_r)$ denote the flag variety of signature $(n_1, \dots, n_r)$ in $n_r$-space. The functor represented by $\Fl(n_1, \dots, n_r)$ sends a commutative $k$-algebra $A$ to isomorphism classes of diagrams
\begin{equation}\label{eq:FlFOP}
    \begin{tikzcd}
        Q_r = A^{n_r} \rar["q_{r-1}"] & Q_{r-1} \rar[ "q_{r-2}"] & \cdots \rar["q_2"] & Q_2 \rar["q_1"] & Q_1
    \end{tikzcd}
\end{equation}
such that, for each $i$, the map $q_i$ is an $A$-module epimorphism and $Q_i$ is a projective $A$-module of rank $n_i$. An isomorphism of two such diagrams is a collection of $A$-module isomorphisms $h_i\co Q_i \to Q_i'$, with $h_r = \id$, that form a commuting ladder
\begin{equation}\label{eq:FlFOPIso}
    \begin{tikzcd}
        A^{n_r} \rar["q_{r-1}"] \dar[equals] & Q_{r-1} \rar["q_{r-2}"] \dar["h_{r-1}","\cong"'] & \cdots \rar["q_2"] & Q_2 \rar["q_1"] \dar["h_2","\cong"'] & Q_1 \dar["h_1","\cong"'] \\
        A^{n_r} \rar["q_{r-1}'"] & Q_{r-1}' \rar["q_{r-2}'"] & \cdots \rar["q_2'"] & Q_2' \rar["q_1'"] & Q_1' \mathrlap{\; .}
    \end{tikzcd} 
\end{equation}
An important special case is $\Fl(m,n) = \Gr(m,n)$. We allow for $n_1=0$, in which case we identify $\Fl(0,n_2, \dots, n_r) = \Fl(n_2, \dots, n_r)$.

Let $\Fl'(n_1, \dots, n_r)$ denote the $k$-scheme that represents the functor that sends a $k$-algebra $A$ to isomorphism classes of tuples $(q_i,s_i)_{i=1}^{r-1}$, where $q_i\co  Q_{i+1} \to Q_i$ are surjections as in \eqref{eq:FlFOP} and $s_i\co Q_i \to Q_{i+1}$ is a section of $q_i$ for each $i$. An isomorphism from $(q_i,s_i)_i$ to $(q_i',s_i')_i$ is a collection of $A$-module isomorphisms $h_i\co Q_i \to Q_i'$, with $h_r = \id$, such that $h_iq_i = q_i'h_{i+1}$ and $h_{i+1}s_i = s_i'h_i$. We define $\Gr'(m,n) := \Fl'(m,n)$. 

Next, define $V'(m,n)$ to the be affine subvariety of $\Mat_{m \times n} \times \Mat_{n \times m}$ consisting of pairs of matrices $(q,s)$ such that $qs=I_m$, the $m \times m$ identity matrix. There is a map $V'(m,n) \to \Gr'(m,n)$ which on $A$-points sends a pair $(q,s)$ to its isomorphism class. Projection to the first factor yields a map $\pr_1\co  V'(m,n) \to V(m,n)$, where $V(m,n)$ is the usual Stiefel variety of full-rank $m \times n$ matrices. This descends to a map $\overline{\pr}_1\co  \Gr'(m,n) \to \Gr(m,n)$, and these maps fit into a pullback square
\[
    \begin{tikzcd}
        V'(m,n) \dar \rar["\pr_1"] \ar[phantom, dr, pos=0, "\lrcorner"] & V(m,n) \dar \\
        \Gr'(m,n) \rar["\overline{\pr}_1"] & \Gr(m,n)
    \end{tikzcd}
\]
in $\cat{Sm}_k$. Recall that the standard affine cover of $V(m,n)$ consists of (Zariski) opens $V_I \subseteq V(m,n)$ indexed by sequences $I = (i_1, \dots, i_m)$ with $1 \leq i_1 < \cdots < i_m \leq n$ such that the minor corresponding to $I$ is nonvanishing. Let $U_I \subseteq \Gr(m,n)$ denote the standard affine open corresponding to $I$, defined similarly. The next lemma is straightforward. 

\begin{lem}\label{lem:VAffineBund}
    The projection $\pr_1\co  V'(m,n) \to V(m,n)$ is an affine-space bundle that trivializes over the standard affine cover of $V(m,n)$.
\end{lem}

\begin{lem}\label{lem:GrAffineBund}
    The morphism $\overline{\pr}_1\co  \Gr'(m,n) \to \Gr(m,n)$ is an affine-space bundle that trivializes over the standard affine cover of $\Gr(m,n)$.
\end{lem}
\begin{proof}
    Let $U_I \subseteq \Gr(m,n)$ be a standard affine open. There is a closed embedding $U_I \to V(m,n)$ which trivializes the principal $\GL_m$-bundle $V(m,n) \to \Gr(m,n)$, and over $U_I$ the map $\overline{\pr}_1$ is the restriction of the affine space bundle 
    \[
        \pr_1\co  V'(m,n) \to V(m,n)
    \]
    of \Cref{lem:VAffineBund}.
\end{proof}

\begin{rem}
    Lemmas \ref{lem:VAffineBund} and \ref{lem:GrAffineBund} are examples of Jouanolou's device (\cite{Wei:HomotopyAlgebraicKTheory}*{Sec. 4}).
\end{rem}

There is a map 
\[
    \pi\co  \Fl'(n_1, \dots, n_r) \to \Fl(n_1, \dots, n_r)
\]
which forgets the sections $s_i$. Also, for each $i = 1, \dots, r-1$, there is a morphism $k$-varieties $\Fl(n_1, \dots, n_r) \to \Gr(n_i,n_r)$ which on points sends the diagram \eqref{eq:FlFOP} to the isomorphism class of the surjection
\[
    q_{k-1} \circ \cdots \circ q_i\co  A^{n_r} \to Q_i. 
\]
The product of these maps is a closed immersion 
\[
    j\co \Fl(n_1, \dots, n_r) \to \prod_{i=1}^{r-1} \Gr(n_i,n_r).
\]
There is a completely analogous map 
\[
    j'\co \Fl'(n_1, \dots, n_r) \to \prod_{i=1}^{r-1} \Gr'(n_i,n_r)
\]
which fits into the commuting square
\begin{equation}\label{eq:Fl'PB}
    \begin{tikzcd}
        \Fl'(n_1, \dots, n_r) \rar["j'"] \dar["\pi"] & \prod_i \Gr'(n_i,n_r) \dar["\prod \overline{\pr}_1"] \\
        \Fl(n_1, \dots, n_r) \rar["j"] & \prod_i \Gr(n_i,n_r)
    \end{tikzcd}
\end{equation}
\begin{lem}\label{lem:Fl'PB}
    The diagram \eqref{eq:Fl'PB} is a pullback in $\cat{Sm}_k$.
\end{lem}
\begin{proof}
Since the Yoneda embedding $\cat{Sm}_k \to Pre(\cat{Sm}_k)$ preserves limits, this is an exercise in tracing through the functor of points definitions of these $k$-schemes.
\end{proof}

\begin{cor}\label{cor:FlAffBund}
    The map $\pi\co \Fl'(n_1, \dots, n_r) \to \Fl(n_1, \dots, n_r)$ is a Zariski-locally trivial affine space bundle. In particular, $\pi$ is an $\A^1$-weak equivalence.
\end{cor}

\begin{proof}
    This follows from Lemmas \ref{lem:GrAffineBund} and \ref{lem:Fl'PB} and the fact that a finite product of Zariski-locally trivial affine space bundles is a Zariski-locally trivial affine space bundle. 
\end{proof}

\begin{construction} 
    We define $A(n_1, \dots, n_r)$ to be the functor which sends a commutative $k$-algebra $A$ to isomorphism classes of pairs $((q_i,s_i)_i, g)$ consisting of an $A$-point $(q_i,s_i)_i$ of $\Fl'(n_1, \dots, n_r)$ together with an $A$-module automorphism $g$ of $Q_1$. An isomorphism from $((q_i,s_i),g)$ to $((q_i',s_i'),g')$ is an isomorphism $(h_i)_i$ from $(q_i,s_i)_i$ to $(q_i',s_i')_i$ (in the sense of \eqref{eq:FlFOPIso}) such that $h_1g = g'h_1$. We denote the isomorphism class of $((q_i,s_i)_i,g)$ in $A(n_1, \dots, n_l)(A)$ by $[(q_i,s_i)_i,g]$. Note that when $n_1=0$ we have $A(0,n_2, \dots, n_r) = \Fl'(n_2, \dots, n_r)$.
\end{construction}

\begin{lem}\label{lem:ARepresentable}
    The presheaf $A(n_1, \dots, n_r)$ is represented by a smooth $k$-scheme. 
\end{lem}
\begin{proof}
    That $A(n_1, \dots, n_r)$ is a sheaf for the Zariski topology follows from the fact that projective modules and maps between them glue. We show that $A(n_1, \dots, n_r)$ can be covered by smooth representable functors by induction on $r$. For the base case, we have $A(n) = \GL_n$. Now suppose that $A(n_1, \dots, n_{r-1})$ is represented by a smooth $k$-scheme. There is a map of Zariski sheaves
    \[
        f\co  A(n_1, \dots, n_r) \to \Gr'(n_{r-1},n_r)
    \]
    which on points sends $[(q_i,s_i)_i,g]$ to the isomorphism class of the pair $(q_{r-1},s_{r-1})$ in $\Gr'(n_{r-1},n_r)(A)$. If we let $U_I' := \overline{\pr}_1^{-1}(U_I) \subseteq \Gr'(n_{r-1},n_r)$, then the pullback of $U_I'$ under $f$ is the representable sheaf $U_I' \times A(n_1, \dots, n_{r-1})$. This furnishes a cover by smooth representables.
\end{proof}

\begin{rem}
    The scheme $A(1,n)$ is $\Gr'(1,n) \times \G_m$. There is a map 
    \[
        f^1(1,n)\co  \Gr'(1,n) \times \G_m \to \GL_n
    \]
    studied in \cite{Wil:MotivicCoho} (and \cite{Ron:EndomorphismsofP2} for the case of $\SL_n$) which factors through the motivic space $\Sigma^{1,1} \Gr'(1,n)_+ = (\Gr'(1,n) \times \G_m) / (\Gr'(1,n) \times \{1\})$. The following construction generalizes $f^1(1,n)$.
\end{rem}

\begin{construction}\label{const:f,p}
    For any $j \in \{1, \dots, r-1\}$, we construct a morphism 
    \[
        f^j(n_1, \dots, n_r)\co  A(n_1, \dots, n_r) \to A(n_1, \dots, \underline{n}_j, \dots,  n_r),
    \]
    where the underline denotes an omitted index, as follows. Suppose that the isomorphism class of $((q_i,s_i)_i,g)$ is an $A$-point of $A(n_1, \dots, n_r)$. If $j=1$, note that the splitting map $s_1\co Q_1 \to Q_2$ defines an isomorphsim $\psi\co  Q_2 \to \ker q_1 \oplus Q_1$. Now define an automorphism of $Q_2$ by 
    \[
    \begin{tikzcd}
        \tilde{g} \co  Q_2 \rar["\psi"] & \ker q_1 \oplus Q_1 \rar["id \oplus g"] & \ker q_1 \oplus Q_1 \rar["\psi^{-1}"] & Q_2
    \end{tikzcd}
    \]
    and set $f^1(n_1, \dots, n_r)[(q_i, s_i)_{i=1}^{r-1},g] = [(q_i,s_i)_{i=2}^{r-1}, \tilde{g}]$. If $j > 1$, we send $[(q_i,s_i)_{i=1}^{r-1},g]$ to the isomorphism class of the diagram
    \[
    \begin{tikzcd}
        Q_r = A^{n_r} \rar["q_{r-1}"'] & \cdots \lar[bend right, "s_{r-1}"', end anchor = north east, pos = 0.6] \rar["q_{i+1}"'] & Q_{i+1} \rar["q_{i-1}q_i"'] \lar[bend right, "s_{i+1}"', pos = 0.4] & Q_{i-1} \rar["q_{i-2}"'] \lar[bend right, "s_is_{i-1}"'] & \cdots \rar["q_1"'] \lar[bend right, "s_{i-2}"', pos = 0.6] & Q_1 \lar[bend right, "s_1"'] \ar[loop right, looseness = 5,"g"]
    \end{tikzcd}
    \]
    Straightforward diagram chases show that, in any case, $f^j(n_1, \dots, n_r)(A)$ is well-defined on isomorphism classes and natural in $A$ so defines a map of $k$-schemes.
\end{construction}

\begin{rem}\label{rem:AutBundle}
    Consider the map 
    \[
        p\co A(n_1, \dots, n_r) \to \Fl'(n_1, \dots, n_r),
    \]
    defined on points by $p(A)[(q_i,s_i)_{i=1}^{r-1},g] = [(q_i,s_i)_{i=1}^{r-1}]$. If $T$ denotes the tautological rank-$n_1$ bundle on $\Fl(n_1, \dots, n_r)$, then $p$ is the automorphism bundle associated with $\pi^*T$. The map $p$ thus has the structure of an algebraic bundle of groups for which 
    \[
        f^1(0,n_1, \dots, n_r)\co  A(0,n_1, \dots, n_r) = \Fl'(n_1, \dots, n_r) \to A(n_1, \dots, n_r)
    \]
    is the identity section. For this reason we will sometimes write ``$s_{\id}$'' in place of $f^1(0,n_1, \dots, n_r)$. The vector bundle $T$ and the affine-space bundle $\pi$ trivialize over the same cover of $\Fl(n_1, \dots, n_r)$ by affine spaces, so the bundle $p$ trivializes over a cover of $\Fl'(n_1, \dots, n_r)$ by affine spaces.
\end{rem}

\section{Decomposing the motive of \texorpdfstring{$A(n_1, \dots, n_r)$}{A(n1,...,nr)}}\label{sec:MotiveOfA}

The main result of this section is \Cref{prop:APureTate}, which shows that the motive $M(A(n_1, \dots, n_r))$ admits a certain Tate sum decomposition in $\cat{DM}(k,R)$. This is accomplished by way of a motivic version of the Leray-Hirsch theorem in topology (see \Cref{prop:MotivicLH}). 

Suppose $F \to E \xrar{f} B$ is a Zariski-locally trivial algebraic fibre bundle, and let $\alpha \in \h^{p,q}(E)$. Given any Zariski open $U$ in $B$, the class $\alpha|_{f^{-1}(U)}$ defines a map 
\[
    \alpha|_{f^{-1}(U)}\co  M(f^{-1}(U)) \to R(q)[p]
\]
in $\cat{DM}(k,R)$. Let $\phi_{\alpha,U}$ denote the composite
\[
    \begin{tikzcd}
        M(f^{-1}(U)) \rar["\Delta"] &[-1em] M(f^{-1}(U)) \otimes M(f^{-1}(U)) \rar["f_* \otimes (\alpha|_{f^{-1}(U)})"] &[3em] M(U) \otimes R(q)[p] = M(U)(q)[p].
    \end{tikzcd}
\]

The following motivic version of the Leray-Hirsch theorem is modeled on \cite{Pen:MWMotivicCoho}*{Lem. 2.4} and generalizes the projective bundle theorem \cite{MVW:MotivicCoho}*{Thm. 15.12}.

\begin{prop}[A Motivic Leray-Hirsch Theorem]\label{prop:MotivicLH}
    Suppose 
    \[
        \begin{tikzcd}
            F \rar & E \rar["f"] & B
        \end{tikzcd}
    \]
    is a Zariski-locally trivial fibre bundle in $\cat{Sm}_k$ with $F$ connected and $M(F)$ pure Tate. Suppose further that there are classes $\alpha_1, \dots, \alpha_n$, with $\alpha_i \in \h^{p_i,q_i}(E)$, and a finite Zariski cover $\{U_j\}$ of $B$ such that for all $j$ and all open $V \subseteq U_j$, the map
    \[
        \oplus_i \phi_{\alpha_i,V} \co  M(f^{-1}(V)) \to \bigoplus_{i=1}^n M(V)(q_i)[p_i]
    \]
    is an isomorphism. Then
    \[
        \oplus_i \phi_{\alpha_i,E}\co  M(E) \to \bigoplus_{i=1}^n M(B)(q_i)[p_i]
    \]
    is an isomorphism. Under this identification, the structure map $f$ induces the projection onto the factor $M(B)(0)[0]$.
\end{prop}
\begin{proof}
    The maps $\phi_{\alpha_i,U}$ are natural in $U$. Consider the map of Mayer-Vietoris triangles
    \[
        \begin{tikzcd}[column sep = .9em]
            M(f^{-1}(U \cap V)) \dar["\oplus_i \phi_{\alpha_i,U \cap V}"] \rar & M(f^{-1}(U)) \oplus M(f^{-1}(V))\dar["(\oplus_i \phi_{\alpha_i,U}) \oplus (\oplus_i \phi_{\alpha_i,V})"] \rar & M(f^{-1}(U \cup V) \rar["{[1]}"] \dar["\oplus_i \phi_{\alpha_i,U \cup V}"] & \! \\
            \bigoplus_i M(U \cap V)(q_i)[p_i] \rar & (\bigoplus_i M(U)(q_i)[p_i]) \oplus (\bigoplus_i M(V)(q_i)[p_i]) \rar & \bigoplus_i M(U \cup V)(q_i)[p_i] \rar["{[1]}"] & \!
        \end{tikzcd}
    \]
    where the rows are distinguished triangles in $\cat{DM}(k,R)$. If the left two vertical maps are isomorphisms, so is the right-hand vertical map. The result follows from induction on the number of open sets in the cover. 
\end{proof}
\begin{rem}\label{rem:AffineTrivs}
    We will apply this theorem in the case that the algebraic fibre bundle trivializes over a cover of $B$ by affine spaces $U_j = \A^m$. In this case, we have identifications $M(f^{-1}(U_j)) = M(\A^m) \otimes M(F) = M(F)$, and we only need to check that the maps 
    \begin{equation}\label{eq:phiUj}
        \oplus_i \phi_{\alpha_i,U_j}\co  M(F) \to \bigoplus_i M(\A^m)(q_i)[p_i] = \bigoplus_i R(q_i)[p_i]
    \end{equation}
    are isomorphisms for each $j$. Indeed, if $V \subseteq U_j=\A^m$, then the map
    \[
        \oplus_i \phi_{\alpha_i,V}\co M(V) \otimes M(F) \to M(V) \otimes \bigoplus_i R(q_i)[p_i]
    \]
    coincides with $\id \otimes (\oplus_i \phi_{\alpha_i,U_j})$, which is an isomorphism provided \eqref{eq:phiUj} is. Moreover, \eqref{eq:phiUj} is an isomorphism if the classes $\alpha_j|_{f^{-1}(U_j)}$ form a basis for the free $\m_R$-module $\h^{*,*}(f^{-1}(U_j),R) = \h^{*,*}(F,R)$.
\end{rem}

The next task is to show that the flag varieties are pure Tate. The standard argument is to observe that there is a cellular decomposition of flag varieties given by Schubert varieties (see e.g. \cite{K:MotivicCohoOfCellularVarieties}), but this requires resolution of singularities for the base field $k$, as the Schubert varieties are in general not smooth. We sketch an argument that avoids the resolution-of-singularities assumption.

Let $0 < m < n$, and recall that there are two closed immersions $i_1\co \Gr(m,n-1) \to \Gr(m,n)$ and $i_2\co \Gr(m-1,n-1) \to \Gr(m,n)$ defined on $k$-rational points as follows. The morphism $i_1$ is induced by the inclusion $k^{n-1} \to k^n$ given by $\bar{v} \mapsto (\bar{v},0)$, and the morphism $i_2$ sends an $(m-1)$-plane $V \subseteq k^{n-1}$ to $V \oplus \langle e_n \rangle$, where $e_n \in k^n$ is the $n$-th standard basis vector. The closed immersion $i_2$ has codimension $n-m$, and the open complement is an affine space bundle over $\Gr(m,n-1)$. The localization triangle associated with $i_2$ is thus
\[
    M(\Gr(m,n-1)) \xrar{{i_1}_*} M(\Gr(m,n)) \to M(\Gr(m-1,n-1))(n-m)[2(n-m)] \xrar{[1]} .
\]
An inductive argument using \Cref{lem:CohoImpliesSplit} and the fact that the map in motivic cohomology induced by $i_1$ is surjective (a Chow ring calculation) establishes the following lemma. Given an $m$-tuple of integers $\bm{\lambda} = (\lambda_1, \dots, \lambda_m)$, let $N(\bm{\lambda}):= \sum_i \lambda_i$.

\begin{lem}\label{lem:GrPureTate}
    The motive $M(\Gr(m,n))$ admits a Tate sum decomposition
    \[
        M(\Gr(m,n)) = \bigoplus_{\bm{\lambda} \in \Lambda} R(N(\bm{\lambda}))[2N(\bm{\lambda})]
    \]
    where $\Lambda$ is the set of $m$-tuples $\bm{\lambda} = (\lambda_1, \dots, \lambda_m)$ with $\lambda_1 \leq \cdots \leq \lambda_m \leq n-m$.
\end{lem}

\begin{lem}\label{lem:FlPureTate}
    The motive $M(\Fl'(n_1, \dots, n_r))$ is pure Tate. In particular, there is an isomorphism
    \[
        M(\Fl'(n_1, \dots, n_r)) \cong M(\Gr(n_1,n_2)) \otimes M(\Gr(n_2,n_3)) \otimes \cdots \otimes M(\Gr(n_{m-1},n_r)).
    \]
    in $\cat{DM}(k,R)$. 
\end{lem}
\begin{proof}
    It suffices to prove the result for $\Fl(n_1, \dots, n_r)$ by \Cref{cor:FlAffBund}. We proceed by induction on $r$, the base case being $r=2$ which is \Cref{lem:GrPureTate}. For the inductive step, let $f$ denote the structure map of the Grassmannian $n_1$-plane bundle of the Tautological rank-$n_2$ vector bundle on $\Fl(n_2, \dots, n_r)$. The total space is $\Fl(n_1, \dots, n_r)$, so we have a Zariski-locally trivial fibre bundle
    \[
        \Gr(n_1,n_2) \to \Fl(n_1, \dots, n_r) \xrar{f} \Fl(n_2, \dots, n_r)
    \]
    that trivializes over a cover of $\Fl(n_2, \dots, n_r)$ by affine spaces. The tautological rank-$n_1$ vector bundle $T$ on the total space $\Fl(n_1, \dots, n_r)$ and the universal quotient bundle $Q$ on $\Fl(n_1, \dots, n_r)$ provide Chern classes (with $R$-coefficients)
    \[
        c_i(T),c_i(Q) \in \h^{2*,*}(\Fl(n_1, \dots, n_r),R) = \CH^*(\Fl(n_1, \dots, n_r),R).
    \]
    Let $U = \A^m$ be any affine neighborhood that trivializes $f$. Under the map in Chow groups 
    \[
        \CH^*(\Fl(n_1, \dots, n_r),R) \to \CH^*(\Gr(n_1,n_2) \times \A^m,R) = \CH^*(\Gr(n_1,n_2),R)
    \]
    induced by the inclusion $f^{-1}(U) \to \Fl(n_1, \dots, n_r)$, the Chern classes $c_i(T),c_i(Q)$ restrict to Chern classes of the tautological and universal quotient bundles (respectively) on $\Gr(n_1,n_2)$, certain products of which form a basis for the free $R$-module $\CH^*(\Gr(n_1,n_2),R)$. As the Tate summands of $M(\Gr(n_1,n_2))$ are concentrated in Chow height 0, these products of restrictions of $c_i(T),c_i(Q)$ also form a basis for the free $\m_R$-module $\h^{*,*}(\Gr(n_1,n_2),R)$. The result follows from \Cref{prop:MotivicLH} and \Cref{rem:AffineTrivs}. 
\end{proof}
Let $F$ denote the composition
\[
    \begin{tikzcd}[row sep = huge]
        A(n_1, \dots, n_r) \rar["{f^1(n_1, \dots, n_r)}"] &[3em] A(n_2, \dots, n_r) \rar["{f^1(n_2, \dots, n_r)}"] &[3em] \cdots \rar["{f^1(n_{r-1},n_r)}"] &[3em] A(n_r) = \GL_{n_r}.
    \end{tikzcd}
\]

\begin{lem}\label{lem:FibInc}
    For the fibre inclusion inclusion $j_x \co  \GL_{n_1} \to A(n_1, \dots, n_r)$ over any $k$-rational point $x\co  \Spec k \to \Fl'(n_1, \dots, n_r)$, the composition $F \circ j_x$ is na\"{i}vely $\A^1$-homotopic to the inclusion
    \begin{align*}
        \GL_{n_1} & \to \GL_n \\
        A & \mapsto 
        \begin{bmatrix}
            A & 0\\
            0 & I_{n_r-n_1}
        \end{bmatrix}.
    \end{align*}
\end{lem}
\begin{proof}
    The map $F \circ j_x$ is given on points by 
    \[
        A \mapsto 
        h
        \begin{bmatrix}
            A & \\
            0 & I_{n_r-n_1}
        \end{bmatrix}
        h^{-1}
    \] 
    for some $k$-rational point $h \in \GL_{n_r}(k)$. We may assume $h \in \SL_{n_r}(k)$ in which case $h$ is a product of elementary matrices, and conjugation by an elementary matrix is na\"{i}vely $\A^1$-homotopic to the identity map on $\GL_{n_r}$.
\end{proof}

We record here the motivic cohomology ring of the Stiefel varieties $V(m,n)$ (and in particular $\GL_n$). The calculation can be found in \cite{Pus:HigherChernClasses} for the general linear groups and in \cite{Wil:MotivicCoho} for all Stiefel varieties. Let $\{u\}$ denote the image of  $u \in k^\times = \m_\Z^{1,1}$ under the map $\m_\Z \to \m_R$ induced by the map of coefficient rings.

\begin{lem}[O. Pushin, B. Williams]\label{lem:VCoho}
    The ring $\h^{*,*}(V(m,n),R)$ admits a presentation as the polynomial algebra over $\m_R$ generated by elements $\alpha_i$ in bidegree $(2i-1,i)$, where $i \in \{n-m+1, \dots, n\}$, subject to the relations 
    \[
	\alpha_i^2 = 
        \begin{cases}
		  0 & \text{if } 2i-1 > n, \\
		  \{-1\}\alpha_{2i-1} & \text{otherwise}
		\end{cases}
    \]
    as well as the relations imposed by the motivic cohomology ring being graded commutative in the first grading and commutative in the second. 
\end{lem}

In particular, the product classes $\rho_{i_1} \cdots \rho_{i_l}$, for $n-m < i_1 < \cdots < i_l \leq n$ form a basis for the free $\m_R$-module $\h^{*,*}(V(m,n),R)$. 

We now come to the main result of this section.

\begin{prop}\label{prop:APureTate}
    The motive $M(A(n_1, \dots, n_r))$ is pure Tate. In particular, there is a decomposition
    \[
        M(A(n_1, \dots, n_l)) = M(\GL_{n_1}) \otimes M(\Fl(n_1, \dots, n_r))
    \]
    such that, under the identification of \Cref{prop:GLnPureTate}, the map $p\co A(n_1, \dots, n_r) \to \Fl'(n_1, \dots, n_r)$ induces the projection onto the factor $R(0)[0] \otimes M(\Fl(n_1, \dots, n_r))$.
\end{prop}

\begin{proof}
    We apply \Cref{prop:MotivicLH} to the bundle of groups
    \[
        p\co A(n_1, \dots, n_r) \to \Fl'(n_1, \dots, n_r)
    \]
    with typical fibre $\GL_{n_1}$ that trivializes over a cover by affine spaces (see \Cref{rem:AutBundle}). Note that any $k$-rational point $x\co  \Spec k \to U \cong \A^m$ in a trivializing neighborhood $U \subseteq \Fl'(n_1, \dots, n_r)$ furnishes an $\A^1$-weak equivalence
    \[
        \GL_{n_1} \to p^{-1}(U) \cong \GL_{n_1} \times \A^m.
    \]
    In light of \cite{Wil:MotivicCoho}*{Prop. 8} and \Cref{lem:FibInc}, the set of cohomology classes $\{F^*(\rho_{i_1} \cdots \rho_{i_m})\}_{i_m \leq n_1}$ in $\h^{*,*}(A(n_1, \dots, n_r),R)$ restricts to a basis for the free $\m_R$-module $\h^{*,*}(\GL_{n_1},R)$ under the map in cohomology induced by the inclusion $p^{-1}(U) \to A(n_1, \dots, n_r).$
\end{proof}

\section{Rothenberg--Steenrod spectral sequence calculations}\label{sec:SSCalcs}

In this section, we study the Rothenberg--Steenrod Spectral sequence central to the proof of splitting of the motive of $\GL_n$. To this end, we exhibit the $k$-scheme $A(n_1, \dots, n_r)$ as the base of a certain principal $\GL_{n_{r-1}}$-bundle. We assume in this section that the pair $(k,R)$ satisfies Beilinson-Soul\'{e} vanishing.

We now describe the group action of interest. There is an action
\[
    a\co \GL_{n_{r-1}} \times A(n_1, \dots, n_{r-1}) \to A(n_1, \dots, n_{r-1})
\]
defined as follows. If $r=2$, we give $A(n) = \GL_n$ the adjoint action. For $r>2$, let $h \in \GL_{n_{r-1}}(A)$ and define the action on points by
\[
    h.[((q_i,s_i)_{i=1}^{r-2},g)] := [((q_i',s_i')_{i=1}^{r-2},g)]
\]
where 
\[
    (q_i',s_i') := 
    \begin{cases}
        (q_i,s_i) & \text{if } i < r-2, \\
        (q_{r-2}h^{-1},hs_{r-2}) & \text{if } i=r-2.
    \end{cases}
\]
A diagram chase shows that the map
\[
    f^1(n_1, \dots, n_{r-1})\co  A(n_1, \dots, n_{r-1}) \to A(n_2, \dots, n_{r-1})
\]
of Construction \ref{const:f,p} is $\GL_{n_{r-1}}$-equivariant. 

The group scheme $\GL_{n_{r-1}}$ acts freely on $V'(n_{r-1},n_r)$ on the left by $h.(q,s) = (hq,sh^{-1})$. If we endow $V'(n_{r-1},n_r) \times A(n_1, \dots, n_{r-1})$ with the diagonal $\GL_{n_{r-1}}$-action, then there is a $\GL_{n_{r-1}}$-equivariant morphism 
\[
    \pi\co  V'(n_{r-1},n_r) \times A(n_1, \dots, n_{r-1}) \to A(n_1, \dots, n_r),
\]
where $\GL_{n_{r-1}}$ acts trivially on the target, defined as follows. Suppose we have $A$-points $(q,s)$ and $[(q_i,s_i)_i,g]$ of $V'(n_{r-1},n_r)$ and $A(n_1, \dots, n_{r-1})$, respectively. Then the isomorphism class of the diagram
\[
    \begin{tikzcd}
        A^{n_r} \rar["q"'] & A^{n_{r-1}} \rar["q_{r-2}"'] \lar[bend right, "s"'] & Q_{r-2} \rar["q_{r-3}"'] \lar[bend right,"s_{r-2}"'] & \cdots \lar[bend right,"s_{r-3}"', pos = .6] \rar["q_2"'] & Q_2 \lar[bend right,"s_2"'] \rar["q_1"'] & Q_1 \lar[bend right, "s_1"'] \ar[loop right, looseness = 5,"g"]
    \end{tikzcd}
\]
is an $A$-point of $A(n_1, \dots, n_r)$. This assignment defines the morphism $\pi$. To see that $\pi$ is equivariant, suppose $((q,s),[(q_i,s_i)_i,g])$ is an $A$-point of $V'(n_{r-1},n_r) \times A(n_1, \dots, n_{r-1})$ and $h \in \GL_{n_{r-1}}(A)$. When $r > 2$, the isomorphism of diagrams 
\[
\begin{tikzcd}
     A^{n_r} \rar["hq"'] \dar[equals] & A^{n_{r-1}} \dar["h^{-1}"] \rar["q_{r-2}h^{-1}"'] \lar[bend right, "sh^{-1}"'] & Q_{r-2} \dar[equals] \rar["q_{r-3}"'] \lar[bend right,"hs_{r-2}"'] & \cdots \lar[bend right,"s_{r-3}"', pos = .6] \rar["q_2"'] & Q_2 \dar[equals] \lar[bend right,"s_2"'] \rar["q_1"'] & Q_1 \dar[equals] \lar[bend right, "s_1"'] \ar[loop right, looseness = 5,"g"]\\
     A^{n_r} \rar["q"] & A^{n_{r-1}} \rar["q_{r-2}"] \lar[bend left, "s"] & Q_{r-2} \rar["q_{r-3}"] \lar[bend left,"s_{r-2}", pos = .4] & \cdots \lar[bend left,"s_{r-3}", pos = .6] \rar["q_2"] & Q_2 \lar[bend left,"s_2"] \rar["q_1"] & Q_1 \lar[bend left, "s_1"] \ar[loop right, looseness = 5,"g"]
\end{tikzcd}
\]
shows that $\pi(A)((q,s),[(q_i,s_i)_i,g]) = \pi(A)(h.((q,s),[(q_i,s_i)_i,g]))$. And when $r=2$, the isomorphism of diagrams
\[
\begin{tikzcd}
    A^{n_2} \dar[equals] \rar["hq"'] & A^{n_1} \dar["h^{-1}"] \lar[bend right,"sh^{-1}"'] \ar[loop right, looseness = 5, "hgh^{-1}"] \\
    A^{n_2} \rar["q"] & A^{n_1} \lar[bend left, "s"] \ar[loop right, looseness = 5, "g"]
\end{tikzcd}
\]
shows that $\pi$ is $\GL_{n_{r-1}}$-equivariant in this case as well.
\begin{lem}
    The morphism $\pi\co  V'(n_{r-1},n_r) \times A(n_1, \dots, n_{r-1}) \to A(n_1, \dots, n_r)$ is a Zariski-locally trivial principal $\GL_{n_{r-1}}$-bundle. 
\end{lem}
\begin{proof}
    Analyzing points and some diagram chasing shows that the map $(\alpha, \pr_2)$ of Definition \ref{def:PrincBund} is an isomorphism of Nisnevich sheaves. To establish local triviality, recall that the open affines $U_I' \subseteq \Gr'(n_{r-1},n_r)$ trivialize the algebraic fibre bundle $f\co A(n_1, \dots, n_r) \to \Gr'(n_{r-1},n_r)$ (see the proof of \Cref{lem:ARepresentable}). So we have isomorphisms
    \[
        f^{-1}(U_I') \cong U_I' \times A(n_1, \dots, n_{r-1}).
    \]
    The set $U_I'(A)$ consists of pairs $(q,s) \subseteq \Mat_{n_{r-1} \times n_r}(A) \times \Mat_{n_r \times n_{r-1}}(A)$ such that submatrices of $q$ and $s$ obtained by selecting the columns and rows, respectively, that correspond to $I$ is the identity matrix. Such a pair $(q,s)$ lies in $V'(n_{r-1},n_r)(A)$. Hence the inclusions
    \[
        f^{-1}(U_I') = U_I' \times A(n_1, \dots n_{r-1}) \to V'(n_{r-1},n_r) \times A(n_1, \dots, n_{r-1})
    \]
    provide trivializing sections of $\pi$. 
\end{proof}

The map $\pi$ satisfies the assumptions of Proposition \ref{prop:RSSS}, so there is a spectral sequence 
\begin{align*}
    \E_2^{l,p,q} &= \Ext^l_{\widehat{\h^{*,*}(\GL_{n_{r-1}},R)}}(\widehat{\h^{p,q}(V'(n_{r-1},n_r) \times A(n_1, \dots, n_{r-1}),R)}, \m_R) \Rightarrow \h^{l+p,q}(A(n_1, \dots, n_r),R).
\end{align*}
To determine the $\E_2$-page, we need to determine the module structure of 
\[
   \widehat{\h^{*,*}(V'(n_{r-1},n_r) \times A(n_1, \dots, n_{r-1}),R)} 
\]
as a module over $\widehat{\h^{*,*}(\GL_{n_{r-1}},R)}$. Moreover, by \Cref{lem:VCoho} and \Cref{prop:APureTate}, the motivic cohomology rings $\h^{*,*}(A(n_1, \dots, n_{r-1}),R)$ and $\h^{*,*}(V'(n_{r-1},n_r))$ are free, finitely generated $\m_R$-modules. We therefore have a decomposition
\[
    \widehat{\h^{*,*}(V'(n_{r-1},n_r) \times A(n_1, \dots, n_{r-1}),R)} = \widehat{\h^{*,*}(V'(n_{r-1},n_r),R)} \otimes_{\m_R} \widehat{\h^{*,*}(A(n_1, \dots, n_{r-1}),R)}
\]
as $\m_R$-modules, so that we may calculate the module structure of the duals of the motivic cohomology rings of $V'(n_{r-1},n_r)$ and $A(n_1, \dots, n_{r-1})$ independently and assemble them. 

The following lemma is \cite{Wil:EquivariantMotivicCoho}*{Prop. 2.9}. Let $\rho_i$ and $\alpha_i$ denote the $\m_R$-algebra generators of $\h^{*,*}(\GL_n,R)$ and $\h^{*,*}(V(m,n),R)$, respectively. 

\begin{lem}\label{lem:MultMC}
    The multiplication $m\co \GL_n \times \GL_n \to \GL_n$ induces the comultiplication
    \[
        m^*\co  \h^{*,*}(\GL_n) \to \h^{*,*}(\GL_n) \otimes_{\m_R} \h^{*,*}(\GL_n)
    \]
    defined by $\rho_i \mapsto \rho_i \otimes 1 + 1 \otimes \rho_i$ for each $i=1, \dots, n$. 
\end{lem}

We will need the following lemma which can be found in \cite{Wil:MotivicCoho}.

\begin{lem}\label{lem:ProjMC}  
    The map $\GL_n \to V(m,n)$ given by projection onto the first $m$ rows induces the inclusion of motivic cohomology rings defined by $\alpha_i \mapsto \rho_i$ for each $i=n-m+1, \dots, n$. 
\end{lem}

\begin{lem}\label{lem:SVActionMC}
    The action $\GL_m \times V'(m,n) \to V'(m,n)$ induces the map in motivic cohomology
    \[
        \h^{*,*}(V'(m,n)) \to \h^{*,*}(\GL_m) \otimes_{\m_R} \h^{*,*}(V'(m,n))
    \]
    defined by
    \[
	\alpha_i \mapsto 
	\begin{cases}
            \rho_i \otimes 1 + 1 \otimes \alpha_i & \text{if } i \leq m \\
            1 \otimes \alpha_i & \text{otherwise}
	\end{cases}
    \]
\end{lem}
\begin{proof}
    The $\A^1$-weak equivalence $\pr_1 \co  V'(m,n) \to V(m,n)$ of \Cref{lem:VAffineBund} is $\GL_m$-equivarant, so it suffices to establish the statement for the $\GL_m$-action on $V(m,n)$. Identify $\GL_m$ with the closed subscheme of $\GL_n$ given by matrices of the form
    \[
        \begin{bmatrix}
		  g & 0 \\
    	  0 & I_{n-m}
	\end{bmatrix}.
    \]
    The action map $\GL_m \times V(m,n) \to V(m,n)$ fits into the diagram
    \[
    \begin{tikzcd}
	\GL_m \times V(m,n) \rar & V(m,n) \\
		\GL_m \times \GL_n \uar["\id \times \pr"] \dar \rar["m"] & \GL_n \uar \dar[equals] \\
		\GL_n \times \GL_n \rar["m"] & \GL_n 
    \end{tikzcd}
    \]
    where $\pr\co \GL_n \to V(m,n)$ is projection onto the first $m$ rows, and the bottom left vertical map is the inclusion. A diagram chase appealing to \cite{Wil:MotivicCoho}*{Prop. 8}, \Cref{lem:MultMC}, and \Cref{lem:ProjMC} yields the result. 
\end{proof}

The computation of the map in cohomology induced by the action of $\GL_{n_{r-1}}$ on $A(n_1, \dots, n_{r-1})$ is more complicated. We assume \Cref{lem:AAction} holds in order to describe the spectral sequence converging to $\h^{*,*}(A(n_1, \dots, n_r),R)$. We then prove \Cref{lem:AAction} by induction on $r$.

\begin{lem}\label{lem:AAction}
    The action map $a\co \GL_{n_{r-1}} \times A(n_1, \dots, n_{r-1}) \to A(n_1, \dots, n_{r-1})$ induces the trivial coaction in cohomology:
    \begin{align*}
        a^*\co \h^{*,*}(A(n_1, \dots, n_{r-1})) &\to \h^{*,*}(\GL_{n_{r-1}}) \otimes_{\m} \h^{*,*}(A(n_1, \dots, n_{r-1})), \\
        \beta &\mapsto 1 \otimes \beta.
    \end{align*}
\end{lem}

The Tate sum decomposition of \Cref{prop:APureTate} provides a basis $\{\beta_j\}_{j=1}^n$ for the free $\m_R$-module $\h^{*,*}(A(n_1, \dots, n_{r-1}),R)$. Hence, there is a decomposition
\[
    \h^{*,*}(V(n_{r-1},n_r),R) \otimes_{\m_R} \h^{*,*}(A(n_1, \dots, n_{r-1}),R) = \bigoplus_{j=1}^n \h^{*,*}(V'(n_{r-1},n_r),R) \cdot \beta_j
\]
as an $\m_R$-module. In particular, the elements $\alpha_i \otimes \beta_j$ form a set of $\m_R$-algebra generators, so to describe the map in cohomology induced by the action
\[
    \GL_{n_{r-1}} \times (V'(n_{r-1},n_r) \times A(n_1, \dots, n_{r-1})) \to V'(n_{r-1},n_r) \times A(n_1, \dots, n_{r-1}),
\]
it suffices to describe the effect on the algebra generators $\alpha_i \otimes \beta_j$. Identify
\begin{multline*}
    \h^{*,*}(\GL_{n_{r-1}} \times V'(n_{r-1},n_r) \times A(n_1, \dots, n_{r-1}),R) \\
    =\h^{*,*}(\GL_{n_{r-1}},R) \otimes_{\m_R} \h^{*,*}(V(n_{r-1},n_r),R) \otimes_{\m_R} \h^{*,*}(A(n_1, \dots, n_{r-1}),R)
\end{multline*}
via the K\"{u}nneth formula. 

\begin{lem}\label{lem:DiagAction}
    Assume that \Cref{lem:AAction} holds. Then the action map
    \[
        \GL_{n_{r-1}} \times (V'(n_{r-1},n_r) \times A(n_1, \dots, n_{r-1})) \to V'(n_{r-1},n_r) \times A(n_1, \dots, n_{r-1})
    \]
    induces the map in motivic cohomology given on $\m_R$-algebra generators by
    \[
        \alpha_i \otimes \beta_j \mapsto
        \begin{cases}
            \rho_i \otimes 1 \otimes \beta_j + 1 \otimes \alpha_i \otimes \beta_j & \text{if } i \leq n_{r-1} \\
            1 \otimes \alpha_i \otimes \beta_j & \text{otherwise}
        \end{cases}
    \]
\end{lem}
\begin{proof}
    This follows by comparing to the action maps
    \begin{align*}
        \GL_{n_{r-1}} \times V'(n_{r-1},n_r) &\to V'(n_{r-1},n_r), \\
        \GL_{n_{r-1}} \times A(n_1, \dots, n_{r-1}) &\to A(n_1, \dots, n_{r-1})
    \end{align*}
    and applying Lemmas \ref{lem:SVActionMC} and \ref{lem:AAction}.
\end{proof}

We may now describe the $\E_2$-page of the spectral sequence converging to $\h^{*,*}(A(n_1, \dots, n_r),R)$. 

\begin{lem}\label{lem:E2}
    Assume that \Cref{lem:AAction} holds. The $\E_2$-page of the motivic Rothenberg--Steenrod spectral sequence associated with the $\GL_{n_{r-1}}$-action on $V'(n_{r-1},n_r) \times A(n_1, \dots, n_{r-1})$, converging to $\h^{*,*}(A(n_1, \dots, n_r),R)$, admits the following presentation as a trigraded $\m_R$-module:
    \[
        \E_2^{*,*,*} = \bigoplus_{j=1}^n \Lambda_{\m_R}(\alpha_{n_r-n_{r-1}+1}', \dots, \alpha_{n_r}')[\theta_1, \dots, \theta_{n_{r-1}}] \cdot \beta_j'
    \]
    where $\alpha_i'=0$ if $i \leq n_{r-1}$, the generators $\alpha_i'$ and  $\theta_i$ have tridegrees $(0,2i-1,i)$ and $(1, 2i-1,i)$, respectively, and the elements $\beta_j'$ corresponds to the $\m_R$-module generators $\beta_j$ of $\h^{*,*}(A(n_1, \dots, n_{r-1}),R)$. If $\beta_j$ has bidegree $(p_j,q_j)$, then $\beta_j'$ has tridegree $(0,p_j,q_j)$.
\end{lem}
\begin{proof}
    Given a generator $\alpha$ of a free $\m_R$-module, let $\widehat{\alpha}$ denote the class dual to $\alpha$. Define 
    \begin{align*} 
        S &:= \h^{*,*}(\GL_{n_{r-1}},R) \\
        M &:= \h^{*,*}(V(n_{r-1},n_r) \times A(n_1, \dots, n_{r-1}),R) = \bigoplus_j \h^{*,*}(V(n_{r-1},n_r),R) \cdot \beta_j.
    \end{align*}
    We first note that, although the cohomology ring $S$ is not an exterior algebra in general (see \Cref{lem:VCoho}), it follows readily from \Cref{lem:MultMC} that the ring $\widehat{S}$ is an honest exterior algebra over $\m_R$:
    \[
        \widehat{S} = \Lambda_{\m_R}(\widehat{\rho}_1, \dots, \widehat{\rho}_{n_{r-1}}).
    \]
    By \Cref{lem:DiagAction}, we have a decomposition
    \[
        \widehat{M} = \bigoplus_j \widehat{\h^{*,*}(V(n_{r-1},n_r),R)} \cdot \widehat{\beta}_j
    \]
    as an $\widehat{S}$-module. So 
    \begin{align*}
        \E_2^{*,*,*} = \Ext^*_{\widehat{S}}(\widehat{M},\m_R) &= \Ext_{\widehat{S}}^* \Big( \bigoplus_j \widehat{\h^{*,*}(V(n_{r-1},n_r),R)} \cdot \widehat{\beta}_j, \m_R \Big) \\
        &= \bigoplus_j \Ext_{\widehat{S}}^* ( \widehat{\h^{*,*}(V(n_{r-1},n_r),R)}, \m_R) \cdot \widehat{\widehat{\beta_j}}
    \end{align*}
    Set $\beta_j' = \widehat{\widehat{\beta_j}}$. Then $\beta_j'$ has the claimed tridegree. The result now follows from the homological algebra calculation
    \[
        \Ext_{\widehat{S}}^*(\widehat{\h^{*,*}(V(n_{r-1},n_r),R)}, \m_R) = \Lambda_{\m_R}(\alpha_{n_r-n_{r-1}+1}', \dots, \alpha_{n_r}')[\theta_1, \dots, \theta_{n_{r-1}}],
    \]
    with $\alpha_i'$ and $\theta_i$ in the claimed tridegrees, which is \cite{Wil:EquivariantMotivicCoho}*{Prop. A.7}. We remark that in the statement \cite{Wil:EquivariantMotivicCoho}*{Prop. A.7}, the assumption that $S$ and $M$ are exterior $\m_R$-algebras can be relaxed as long as the Hopf algebra structure makes $\widehat{S}$ an exterior algebra over $\m_R$ and $\widehat{M}$ decomposes into a direct sum of $\widehat{S}$-modules of the form $\widehat{S}/(\widehat{\rho}_{i_1}, \dots,  \widehat{\rho}_{i_l})$, as happens in this case.  
\end{proof}

\begin{figure}
\[
    \begin{sseqpage}[classes = {draw = none}, xscale = 2, yscale = .8, no x ticks, no y ticks, x label = {$l$}, y label = {$q$}, y label style = {rotate = -90, yshift = -8em, xshift = 5em}, right clip padding = 0.4cm, y axis gap = 3em, lax degree]
	\begin{scope}[background, font = \tiny]
		\node at (0,1) {0};
		\node at (0,2) {1};
		\node at (0,3) {2};
		\node at (0,4) {\vdots};
		\node at (0,5) {q_j};
		\node at (0,6) {q_{j'}};
		\node at (0,7) {\vdots};
		\node at (0,8) {n_r-n_{r-1}+1};
		\node at (0,9) {\vdots};
		\node at (0,10) {n};
		
		\node at (1,0) {0};
		\node at (2,0) {1}; 
		\node at (3,0) {2};
		\node at (4,0) {\cdots};
	\end{scope}

        \class["\vdots"](1,2)
	
	\class["\beta_{j}'"](1,5)
        \class["\beta_{j'}'"](1,6)
	 
	\class["\alpha_{n_r-n_{r-1}+1}'"](1,8)
	\class["\vdots"](1,9)
	\class["\alpha_{n_r}'"](1,10)
	
	\class["\theta_1"](2,2)
	\class["\theta_2"](2,3) 
	\class["\vdots"](2,4)
	
	\class["{\theta_1}^2"](3,3)
	\class["\theta_1\theta_2"](3,4)
	\class["\iddots"](4,4)
	\class["\iddots"](4,5)
	
	\class["1"](1,1)
	\class["0"](2,1)
	\class["0"](3,1)
	\class["0"](3,2) 
	\class["\cdots"](4,2)
	\end{sseqpage}
\]
\caption{Generators of the $\E_2$-page of the Rothenberg--Steenrod spectral sequence converging to $\h^{*,*}(A(n_1, \dots, n))$. The grading $p$ and many product classes are suppressed for clarity.}
\label{fig:SSforA}
\end{figure}

A special case of \Cref{lem:E2} that we will consider is the spectral sequence associated with the $\GL_{n_{r-1}}$-action on 
\[
    V'(n_{r-1},n_r) \times A(0,n_1, \dots, n_{r-1}) = V'(n_{r-1},n_r) \times \Fl'(n_1, \dots, n_{r-1}),
\]
which converges to $\h^{*,*}(\Fl(n_1, \dots, n_r),R)$; we denote this spectral sequence $\{ \tilde{\E}{\vphantom{\E}}^{*,*,*}_s, \tilde{d}_s\}$. The coaction in cohomology induced by action of $\GL_{n_{r-1}}$ on $\Fl(n_1, \dots, n_{r-1})$ is trivial for degree reasons. That is, \Cref{lem:AAction} holds for this coaction (no induction argument is needed).

The map of motives
\[
    M(\Fl'(n_1, \dots, n_{r-1})) \xrar{{s_{\id}}_*} M(A(n_1, \dots, n_{r-1}))
\]
is split by $p_*$, so \Cref{lem:IncTateSum} applies, and we may identify $\h^{*,*}(\Fl'(n_1, \dots, n_{r-1}),R)$ as the free $\m_R$-submodule of $\h^{*,*}(A(n_1, \dots, n_{r-1}),R)$ generated by the classes $\beta_j$ in Chow height 0. Order the generators $\beta_j$ of $\h^{*,*}(A(n_1, \dots, n_{r-1}),R)$ so that $\beta_1, \dots, \beta_m$ have Chow height 0, and $\beta_{m+1}, \dots, \beta_n$ have positive Chow height. The $\E_2$-page of the spectral sequence converging to $\h^{*,*}(\Fl'(n_1, \dots, n_r),R)$ is then given by
\[
    \tilde{\E}{\vphantom{\E}}_2^{*,*,*} = \bigoplus_{j=1}^m \Lambda_{\m_R}(\alpha_{n_r-n_{r-1}+1}', \dots, \alpha_{n_r}')[\theta_1, \dots, \theta_{n_{r-1}}] \cdot \beta_j'.
\] 

The identity section 
\[
    s_{\id}\co  \Fl'(n_1, \dots, n_r) \to A(n_1, \dots, n_r)
\]
of $p$ is induced by the $\GL_{n_{r-1}}$-equivariant map 
\begin{equation}\label{eq:idxs}
    \id \times s_{\id}\co  V'(n_{r-1},n_r) \times \Fl'(n_1, \dots, n_{r-1}) \to V'(n_{r-1},n_r) \times A(n_1, \dots, n_{r-1}).
\end{equation}
Moreover, the map $p\co A(n_1, \dots, n_r) \to \Fl'(n_1, \dots, n_r)$ is induced by the $\GL_{n_{r-1}}$-equivariant map
\[
    \id \times p\co  V'(n_{r-1},n_r) \times A(n_1, \dots, n_{r-1}) \to V'(n_{r-1},n_r) \times \Fl'(n_1, \dots, n_{r-1}).
\]
These maps participate in the diagram
\begin{equation}\label{eq:SSSplitting}
    \begin{tikzcd}
        V'(n_{r-1},n_r) \times \Fl'(n_1, \dots, n_{r-1}) \dar["\id \times s_{\id}"] \rar["\pi"] \ar[dd,bend right = 40,"\id"', end anchor = north west, start anchor = south west] & \Fl'(n_1, \dots, n_r) \dar["s_{\id}"] \ar[dd,bend left = 40,"\id", start anchor = south east, end anchor  = north east]\\
        V'(n_{r-1},n_r) \times A(n_1, \dots, n_{r-1}) \rar["\pi"] \dar["\id \times p"] & A(n_1, \dots, n_r) \dar["p"] \\
        V'(n_{r-1},n_r) \times \Fl'(n_1, \dots, n_{r-1}) \rar["\pi"] & \Fl'(n_1, \dots, n_r)
    \end{tikzcd}
\end{equation}
which induces a splitting of spectral sequences. Let $S_s\co \E_s^{*,*,*} \to \tilde{\E}{\vphantom{\E}}_s^{*,*,*}$ be the map of spectral sequences induced by $\id \times s_{id}$, and let $P_s \co  \tilde{\E}{\vphantom{\E}}_s^{*,*,*} \to \E_s^{*,*,*}$ be the map induced by $\id \times p$. Assuming \Cref{lem:AAction} holds, we have maps of $\E_2$-pages
\begin{align*}
    S_2\co  \bigoplus_{j=1}^n \Lambda_{\m_R}(\alpha_{n_r-n_{r-1}+1}', \mydots, \alpha_{n_r}')&[\theta_1, \mydots, \theta_{n_{r-1}}] \cdot \beta_j' \\[-1em]
    &\to \bigoplus_{j=1}^m \Lambda_{\m_R}(\alpha_{n_r-n_{r-1}+1}', \mydots, \alpha_{n_r}')[\theta_1, \mydots, \theta_{n_{r-1}}] \cdot \beta_j', \\[-1em] 
    P_2\co  \bigoplus_{j=1}^m \Lambda_{\m_R}(\alpha_{n_r-n_{r-1}+1}', \mydots, \alpha_{n_r}')&[\theta_1, \mydots, \theta_{n_{r-1}}] \cdot \beta_j' \\[-1em]
    &\to \bigoplus_{j=1}^n \Lambda_{\m_R}(\alpha_{n_r-n_{r-1}+1}', \mydots, \alpha_{n_r}')[\theta_1, \mydots, \theta_{n_{r-1}}] \cdot \beta_j'.
\end{align*}
\begin{lem}\label{lem:SSSplitting}
    Assume that \Cref{lem:AAction} holds. The map of $\E_2$-pages $S_2\co  \E_2^{*,*,*} \to \tilde{\E}{\vphantom{\E}}_2^{*,*,*}$ is the projection onto the first $m$ factors, and $P_2\co  \tilde{\E}{\vphantom{\E}}_2^{*,*,*} \to \E_2^{*,*,*}$ is the inclusion into the first $m$ factors.
\end{lem}
\begin{proof}
    We prove the lemma for $S_2$; the proof of the statement for $P_2$ is similar. The map $S_2$ is induced by the map $\id \times s_{\id}$ in \eqref{eq:idxs}, which in cohomology induces the projection
    \[
        (\id \times s_{\id})^*\co  \bigoplus_{i=1}^n \h^{*,*}(V(n_{r-1},n_r),R) \cdot \beta_j \xrar{\pr_{1,\dots,m}} \bigoplus_{i=1}^m \h^{*,*}(V(n_{r-1},n_r),R) \cdot \beta_j
    \]
    The result follows by tracing through the construction of the respective $\E_2$-pages in the proof of \Cref{lem:E2}.
\end{proof}

The next definition is borrowed from \cite{Wil:EquivariantMotivicCoho} and is useful when computing differentials.

\begin{defn}
    For a homogeneous class $\alpha \in \E_s^{l,p,q}$, the \textit{total Chow height} of $\alpha$ is the integer
    \[
		\tch \alpha := 2q-p-l.
    \]
\end{defn}

Note that the differentials decrease the total Chow height by one: 
\[
    \tch d_s(\alpha) = \tch \alpha - 1,
\] 
and for another homogeneous class $\beta \in \E_s^{l',p',q'}$,
\[
    \tch \alpha\beta = \tch \alpha + \tch \beta.
\]

\begin{prop}\label{prop:SSDescription}
    Assume that \Cref{lem:AAction} holds. In the spectral sequence $\{\E_s^{*,*,*}, d_s\}$ converging to $\h^{*,*}(A(n_1, \dots, n_r),R)$, the following hold:
    \renewcommand{\labelenumi}{(\roman{enumi})}
    \begin{enumerate}
        \item for each $i = n_r-n_{r-1}+1, \dots, n_r$, there is some $s$ such that $d_s(\alpha_i')$ is nontrivial, 
        \item for every $j = 1, \dots, n$ and every $s$, the element $d_s(\beta_j')$ is trivial.
    \end{enumerate}
\end{prop}

\begin{proof}
    \textbf{Step 1: proof for $\{ \tilde{\E}{\vphantom{\E}}_s^{*,*,*}, \tilde{d}_s \}$.} The classes $\alpha_i'$ do not support any incoming differentials since $\tilde{\E}{\vphantom{\E}}_s^{l,p,q} = 0$ when $l < 0$. Since $\tch(d_s(\alpha_i')) = 0$, the only class that could possibly support $d_s(\alpha_i')$ is a sum of classes of the form
    \begin{equation}\label{eq:ThetaProd}
        c \theta_{i_1} \cdots \theta_{i_s}
    \end{equation}
    for $c \in \m_R^{0,0} = R$. If $\alpha_i'$ persisted to the $\E_\infty$-page, it would give rise to a free rank-1 summand of the $\m_R$-module $\h^{*,*}(\Fl(n_1, \dots, n_r),R)$ generated by an element in Chow height 1. No such generator exists by the Tate sum decomposition of \Cref{lem:FlPureTate}. We conclude that some differential $d_s$ takes $\alpha_i'$ to a sum of classes of the form \eqref{eq:ThetaProd}, establishing the first part. 
    
    In this case, the classes $\{\beta_j'\}$ all have total Chow height $0$. As Beilinson-Soul\'{e} vanishing holds for the pair $(k,R)$, there are no classes with negative total Chow height, so the differentials vanish on the classes $\beta_j'$.

    We note that the $\GL_{n_{r-1}}$-equivariant projection
    \begin{equation}\label{eq:pr1SSComp}
        \pr_1\co  V'(n_{r-1},n_r) \times \Fl'(n_1, \dots, n_{r-1}) \to V'(n_{r-1},n_r)
    \end{equation}
    induces a map from the spectral sequence associated with the $\GL_{n_{r-1}}$-action on $V'(n_{r-1},n_r)$, which converges to $\h^{*,*}(\Gr(n_{r-1},n_r),R)$, to $\{ \tilde{\E}{\vphantom{\E}}_s^{*,*,*}, \tilde{d}_s \}$. If we let $\bar{\theta}_i$ denote the image of $\theta_i$ in $\tilde{\E}{\vphantom{\E}}_\infty^{*,*,*}$, then the spectral sequence comparison induced by \eqref{eq:pr1SSComp} shows that the classes $\bar{\theta}_i$ and their products generate a free $\m_R$-submodule of $\h^{*,*}(\Fl(n_1, \dots, n_r),R)$ isomorphic to $\h^{*,*}(\Gr(n_{r-1},n_r),R)$. This fact will be used in Step 2.

    \textbf{Step 2: general case.} We compare with the case proved in Step 1. First, we show that the terms $\E_s^{l,p,q}$ and $\tilde{\E}{\vphantom{\E}}_s^{l,p,q}$ can be identified via $S_s$ and $P_s$ whenever $l \geq 1$. The identification $\E_2^{l,p,q} = \tilde{\E}{\vphantom{\E}}_2^{l,p,q}$ when $l \geq 1$ is clear. An element of this term is a sum of classes of the form
    \[
        c \theta_{i_1} \cdots \theta_{i_q}
    \]
    for some $c \in \m_R$. By the Leibniz rule, we see that $d_2\co  \E_2^{l,p,q} \to \E_2^{l+2,p-1,q}$ is trivial when $l \geq 1$, and similarly for $\tilde{d}_2$. Also, since $S_2\co  \E_2^{l,p,q} \to \tilde{\E}{\vphantom{\E}}_2^{l,p,q}$ is surjective for any $l$ by \Cref{lem:SSSplitting}, we must have $\im d_2 = \im \tilde{d}_2$. Consequently, on the succeeding page, $S_3\co  \E_3^{l,p,q} \to \tilde{\E}{\vphantom{\E}}_3^{l,p,q}$ and $P_3\co  \tilde{\E}{\vphantom{\E}}_3^{l,p,q} \to \E_3^{l,p,q}$ are mutually inverse isomorphisms provided $l \geq 1$. An inductive argument shows that $S_s$ and $P_s$ are mutually inverse isomorphisms on those terms with $l \geq 1$ for every $s$.  

    It follows from the splitting $S_2 \circ P_2 = \id$ of $\E_2$-pages established in \Cref{lem:SSSplitting} that $d_s(\alpha_i') = \tilde{d}_s(\alpha_i')$ for each $s$. A portion of the $\E_\infty$-page of the spectral sequence $\{\E_s^{*,*,*},d_s\}$ converging to $\h^{*,*}(A(n_1, \dots, n_r),R)$ is thus
    \[
    \begin{sseqpage}[classes = {draw = none}, xscale = 2, yscale = .8, no x ticks, no y ticks, x label = {$l$}, y label = {$q$}, y label style = {rotate = -90, yshift = -8em, xshift = 5em}, right clip padding = 0.4cm, y axis gap = 3em, lax degree]
	\begin{scope}[background, font = \tiny]
		\node at (0,1) {0};
		\node at (0,2) {1};
		\node at (0,3) {2};
		\node at (0,4) {\vdots};
		\node at (0,5) {i};
		\node at (0,6) {\vdots};
		
		\node at (1,0) {0};
		\node at (2,0) {1}; 
		\node at (3,0) {2};
		\node at (4,0) {\cdots};
	\end{scope}

        \class["\E_\infty^{0,*,1}"](1,2)
        \class["\E_\infty^{0,*,2}"](1,3)
	\class["\E_\infty^{0,*,i}"](1,5)
	
	\class["\bar{\theta}_1"](2,2)
	\class["\bar{\theta}_2"](2,3) 
	\class["\vdots"](2,4)
	
	\class["{\bar{\theta}_1}^2"](3,3)
	\class["\bar{\theta}_1\bar{\theta}_2 "](3,4)
	\class["\iddots"](4,4)
	\class["\iddots"](4,5)
	
	\class["1"](1,1)
	\class["0"](2,1)
	\class["0"](3,1)
	\class["0"](3,2) 
	\class["\cdots"](4,2)
	\end{sseqpage}
    \]
where $\bar{\theta}_i$ is the image of $\theta_i$ on the $\E_\infty$-page. As mentioned in Step 1, the classes $\bar{\theta}_i$ and their products give rise to a free submodule of $\h^{*,*}(A(n_1, \dots, n_r),R)$ isomorphic to $\h^{*,*}(\Gr(n_{r-1},n_r),R)$. \Cref{lem:FlPureTate} and \Cref{prop:APureTate} imply that there is a tensor-product decomposition
\[
    \h^{*,*}(A(n_1, \dots, n_r),R) = \h^{*,*}(\Gr(n_{r-1},n_r),R) \otimes_{\m_R} \h^{*,*}(A(n_1, \dots, n_{r-1}),R)
\]
as $\m_R$-modules. Counting ranks, we must have that the subalgebra $\E_\infty^{0,*,*}$ of $\E_\infty^{*,*,*}$ is a free $\m_R$-module isomorphic to $\h^{*,*}(A(n_1, \dots, n_{r-1}),R)$. We showed above that, for each $i$, the element $d_s(\alpha_i')$ is nontrivial for some $s$. The only possibility is that the classes $\beta_j'$ persist to the $\E_\infty$-page.
\end{proof}

We now focus on tying up the loose end: \Cref{lem:AAction}. The following is \cite{Wil:EquivariantMotivicCoho}*{Cor. 2.9.1}.

\begin{lem}\label{lem:InvMC}
    The inversion map $\imath\co \GL_n \to \GL_n$ induces the map in motivic cohomology defined by $\rho_i \mapsto -\rho_i$ for each $i=1, \dots, n$.
\end{lem}

\begin{lem}\label{lem:AdActionMC}
    The adjoint action 
    \begin{align*}
        \GL_n \times \GL_n & \to \GL_n  \\
        (h,g) & \mapsto hgh^{-1}
    \end{align*}
    induces the map in motivic cohomology 
    \[
		\h^{*,*}(\GL_n) \to \h^{*,*}(\GL_n) \otimes_{\m_R} \h^{*,*}(\GL_n)
    \]
    defined by $\rho_i \mapsto 1 \otimes \rho_i$ for each $i=1, \dots, n$. 
\end{lem}
\begin{proof}
    The action map in question factors as
    \begin{align*}
		\GL_n \times \GL_n \xrar{\Delta \times \id} &\GL_n\times \GL_n \times \GL_n \xrar{\id \times \imath \times \id}  \\
		&\GL_n \times \GL_n \times \GL_n \xrar{\sigma_{23}} \GL_n \times \GL_n \times \GL_n \xrar{m(m \times \id)} \GL_n
    \end{align*}
    where $\Delta$ is the diagonal and $\sigma_{23}$ is the map that swaps the second and third factors. The result now follows by chasing the induced map in cohomology and appealing to Lemmas \ref{lem:MultMC} and \ref{lem:InvMC}. 
\end{proof}

The action of $\GL_{n_{r-1}}$ on $A(n_1, \dots, n_{r-1})$ is induced by the equivariant map
\begin{equation}\label{eq:InducesAAction}
    a' \times \id \co (\GL_{n_{r-1}} \times V'(n_{r-2},n_{r-1})) \times A(n_1, \dots, n_{r-2}) \to V'(n_{r-2},n_{r-1}) \times A(n_1, \dots, n_{r-2})
\end{equation}
where $a'\co \GL_{n_{r-1}} \times V'(n_{r-2},n_{r-1}) \to V'(n_{r-2},n_{r-1})$ is the action given on points by $h.(q,s) = (qh,h^{-1}s)$. Note this is not the same action as the free action of $\GL_{n_{r-1}}$ on $V'(n_{r-1},n_r)$ discussed earlier in this section. 

\begin{lem}\label{lem:a'MC}
    The action $a'\co \GL_{n_{r-1}} \times V'(n_{r-2},n_{r-1}) \to V'(n_{r-2},n_{r-1})$ induces the coaction
    \begin{align*}
        {a'}^*\co \h^{*,*}(V'(n_{r-2},n_{r-1}),R) & \to \h^{*,*}(\GL_{n_{r-1}},R) \otimes_{\m_R} \h^{*,*}(V'(n_{r-2},n_{r-1},R) \\
        \alpha_i &\mapsto \rho_i \otimes 1 + 1 \otimes \alpha_i.
    \end{align*}
\end{lem}
\begin{proof}
    It suffices to prove the result for the action of $\GL_{n_{r-1}}$ on the Stiefel variety $V(n_{r-2},n_{r-1})$ induced by the multiplication map $m\co \GL_{n_{r-1}} \times \GL_{n_{r-1}} \to \GL_{n_{r-1}}$. The result follows by comparison using Lemmas \ref{lem:MultMC} and \ref{lem:ProjMC}.
\end{proof}

If we endow $\GL_{n_{r-1}}$ with the trivial $\GL_{n_{r-2}}$-action, then \eqref{eq:InducesAAction} is $\GL_{n_{r-2}}$-equivariant. The diagram 
\[
    \begin{tikzcd}
        \GL_{n_{r-1}} \times V'(n_{r-2},n_{r-1}) \times A(n_1, \dots, n_{r-2}) \rar["a' \times \id"] \dar["\pi' = \id \times \pi"] & V'(n_{r-2},n_{r-1}) \times A(n_1, \dots, n_{r-2}) \dar["\pi"] \\
        \GL_{n_{r-1}} \times A(n_1, \dots, n_{r-1}) \rar["a"] & A(n_1, \dots, n_{r-1})
    \end{tikzcd}
\]
commutes, where the vertical maps are principal $\GL_{n_{r-2}}$-bundles. Let $\{\tensor*[^\prime]{\E}{_s^{*,*,*}}, \tensor*[^\prime]{d}{_s} \}$ denote the Rothenberg--Steenrod spectral sequence associated with $\pi'$. Then \eqref{eq:InducesAAction} induces a map of spectral sequences $F_s\co \E_s^{*,*,*} \to \tensor*[^\prime]{\E}{_s^{*,*,*}}$.

\begin{proof}[Proof of \Cref{lem:AAction}]
    The proof is by induction on $r$. The base case $r=2$ is \Cref{lem:AdActionMC}. Assume $r > 2$ and that the lemma holds for all sequence $(n_1', \dots, n_{r-1}')$ with $r' < r$. By hypothesis, we may assume that the $\E_2$-page presentation of \Cref{lem:E2} holds for the spectral sequence converging to $\h^{*,*}(A(n_1, \dots, n_{r-1}),R)$. The $\E_2$-page of the spectral sequence $\{ \tensor*[^\prime]{\E}{_s^{*,*,*}}, \tensor*[^\prime]{d}{_s} \}$ admits a presentation
    \[
        \tensor*[^\prime]{\E}{_2^{*,*,*}}= \bigoplus_{j=1}^n \Lambda_{\m_R}(\rho'_1, \dots, \rho'_{n_{r-1}}, \alpha'_{n_{r-1}-n_{r-2}+1}, \dots, \alpha'_{n_{r-1}})[\theta_1, \dots, \theta_{n_{r-2}}] \cdot \beta_j
    \]
    where the tridegrees of $\alpha'_i$ and $\theta_i$ are as in \Cref{lem:E2}, and the elements $\rho'_i$ correspond to algebra generators of $\h^{*,*}(\GL_{n_{r-1}},R)$ and $|\rho_i'| = (0,2i-1,i)$. It follows from \Cref{lem:a'MC} that $F_2(\alpha'_i) = \alpha'_i + \rho'_i$. We also have that $F_2(\beta_j) = \beta_j$ and $F_2(\theta_i) = \theta_i$. The elements $\alpha_i'$ do not persist to the $\E_\infty$-page by \Cref{prop:SSDescription}, and $d_s(\rho_i')=0$ for each $s$. The map $F_\infty$ is thus defined on algebra generators by $F_\infty(\beta_j) = \beta_j$ and $F_\infty(\overline{\theta}_i) = \overline{\theta_i}$. Since products of the elements $\beta_j$ and $\overline{\theta}_i$ give rise to generators of the free $\m_R$-module $\h^{*,*}(A(n_1, \dots, n_{r-1}),R)$, we have established the result. 
\end{proof}

\section{Quotients of \texorpdfstring{$A(n_1, \dots, n_r)$}{A(n1,...,nr)}}\label{sec:AuxSpaces}

For $0 \leq m \leq n_1$, we inductively define a family motivic spaces $X_m(n_1, \dots, n_r)$ as certain homotopy quotients of the schemes $A(n_1, \dots, n_r)$. The constructions of this section are in analogy with Miller's construction of $F_m(n)/F_{m-1}(n)$ as a certain quotient of the space $A(m,n)^u$, as discussed in the introduction. 

First, let $X_0(n_1, \dots, n_r) = A(n_1, \dots, n_r)$. We define the pointed motivic space $X_1(n_1, \dots, n_r)$ to be the $\A^1$-homotopy cofibre of the identity section
\[
    f^1 (0,n_1, \dots, n_r)\co  X_0(0,n_1, \dots, n_r) = \Fl'(n_1, \dots, n_r) \to A(n_1, \dots, n_r) = X_0(n_1, \dots, n_r)
\]
of the map $p$. For any sequence $n_1 < \cdots < n_l < \cdots < n_r$, there is a commuting diagram that encodes all maps
\[
    A(n_1, \dots, n_l, \dots, n_r)  \to A(n_l, \dots, n_r)
\]
obtained by composing the maps $f^j(n_1', \dots, n_r')$ of Construction \ref{const:f,p}. In particular, the square
\[
    \begin{tikzcd}[column sep = 8em]
        X_0(0,n_1, \dots, n_r) \rar["{f^{j+1}(0,n_1, \dots, n_r)}"] \dar["{f^1 (0,n_1, \dots, n_r)}"] & X_0(0,n_1, \dots, \underline{n}_j, \dots, n_r) \dar["{f^1 (0,n_1, \dots, \underline{n}_j, \dots, n_r)}"] \\
        X_0(n_1, \dots, n_r) \rar["{f^j (n_1, \dots, n_r)}"] & X_0(n_1, \dots, \underline{n}_j, \dots, n_r)
    \end{tikzcd}
\]
commutes (recall the nonstandard notion for omitted indices). Taking cofibres of the vertical maps, we have an induced map 
\[
    X_1(n_1, \dots, n_r) \to X_1(n_1, \dots, \underline{n}_j, \dots, n_r) 
\]
which we shall denote $f_1^j(n_1, \dots, n_r)$. The maps $f^j(n_1', \dots, n_r')$ induce a homotopy commuting diagram that encodes all maps
\[
    X_1(n_1, \dots,n_l, \dots, n_r) \to X_1(n_l, \dots, n_r)
\]
obtained by composing the induced maps $f_1^j (n_1', \dots, n_r')$. 

Continuing in this way, suppose we have constructed $X_{m'}(n_1', \dots, n_r')$ for all $m' < m$ and all finite sequences $n_1' < \cdots < n_r'$ with $m' \leq n_1'$, as well as maps
\[
    f_{m'}^j(n_1', \dots, n_r')\co  X_{m'}(n_1', \dots, n_r') \to X_{m'}(n_1', \dots, \underline{n}_j', \dots, n_r').
\]
Suppose further that there are homotopy commuting diagrams that encode all composites
\[
    X_{m'}(n_1', \dots,n_l', \dots, n_r') \to X_{m'}(n_l', \dots, n_r')
\]
of the maps $f_{m'}^j(n_1', \dots, n_r')$. Lastly, suppose that each $f_{m'}^j(n_1', \dots, n_r')$ is the induced map of homotopy cofibres after taking cofibres of the vertical maps in the diagram
\[
    \begin{tikzcd}[column sep = 8em, row sep = large]
        X_{m'-1}(m'-1,n_1', \dots, n_r') \rar["{f_{m'-1}^{j+1}(m'-1,n_1', \dots, n_r')}"] \dar["{f_{m'-1}^1(m'-1,n_1', \dots, n_r')}"] & X_{m'-1}(m'-1,n_1', \dots, \underline{n}_j', \dots, n_r')] \dar["{f_{m'-1}^1(m'-1,n_1', \dots, \underline{n}_j', \dots, n_r')}"] \\
        X_{m'-1}(n_1', \dots, n_r') \rar["{f_{m'-1}^j(n_1', \dots, n_r')}"] & X_{m'-1}(n_1', \dots, \underline{n}_j', \dots, n_r')
    \end{tikzcd}
\]
We define $X_m(n_1, \dots, n_r)$ to be the $\A^1$-homotopy cofibre of
\[  
    X_{m-1}(m-1,n_1, \dots, n_l) \xrar{f_{m-1}^1(m-1,n_1, \dots, n_l)} X_{m-1}(n_1, \dots, n_l).
\]
The diagram 
\[
    \begin{tikzcd}[column sep = 8em, row sep = large]
        X_{m-1}(m-1,n_1, \dots, n_r) \rar["{f_{m-1}^{j+1}(m-1,n_1, \dots, n_r)}"] \dar["{f_{m-1}^1(m-1,n_1, \dots, n_r)}"] & X_{m-1}(m-1,n_1, \dots, \underline{n}_j, \dots, n_r)] \dar["{f_{m-1}^1(m-1,n_1, \dots, \underline{n}_j, \dots, n_r)}"] \\
        X_{m-1}(n_1, \dots, n_r) \rar["{f_{m-1}^j(n_1, \dots, n_r)}"] & X_{m-1}(n_1, \dots, \underline{n}_j, \dots, n_r)
    \end{tikzcd}
\]
commutes. Taking cofibres vertically, we define 
\[
    f_m^j(n_1, \dots, n_r)\co X_m(n_1, \dots, n_r) \to X_m(n_1, \dots, \underline{n}_j, \dots, n_r)  
\]
to be the induced map on cofibres.

\section{Complex realization and motivic cohomology}\label{sec:CxReal}

When $k$ admits a complex embedding, there is a complex realization functor
\[
    Re_\C \co  \cat{Spc}(k) \to \text{Top}
\]
which is left Quillen (where $\text{Top}$ is endowed with the Quillen model structure) and sends a $k$-scheme $Y$ to its associated analytic space $Y(\C)$. We refer to \cite{DI:TopologicalHypercoversRealization} for a detailed treatment. There is a stable version of $Re_\C$ at the level of homotopy categories, which we denote by
\[
    SRe_\C\co  \cat{SH}(k) \to \cat{SH},
\]
constructed in \cite{PPR:VoevodskysKTheorySpectrum}. There, the authors construct $SRe_\C$ at the model category level, but we will not need this generality. 

The following proposition motivates the construction of the motivic spaces $X_m(n_1, \dots, n_r)$. 

\begin{prop}\label{prop:XmRealization}
    If $k$ admits a complex embedding and $m \geq 1$, the spaces $Re_\C(X_m(m,n))$ and $Re_\C(X_m(n))$ have the homotopy type of $F_m(n)/F_{m-1}(n)$ and $U(n)/F_{m-1}(n)$, respectively. The complex realization of $f_m^1(m,n) \co X_m(m,n) \to X_m(n)$ is, up to homotopy, the canonical inclusion
    $F_m(n)/F_{m-1}(n) \to U(n)/F_{m-1}(n)$.
\end{prop}

\begin{proof}
    We omit many of the details, all of which are standard. The statement follows from an induction argument using the fact that $Re_\C$ preserves homotopy cofibres and that there are homeomorphisms 
    \[
        A(m,n)^u/F_{m-1}(m,n) \to F_m(n)/F_{m-1}(n).
    \] 
\end{proof}

Let $HR_{\mot}$ denote the motivic Eilenberg-MacLane spectrum in $\cat{SH}(k)$ associated with $R$ and $HR$ the usual Eilenberg-MacLane spectrum in $\cat{SH}$. 

\begin{lem}
    The complex realization of $HR_{\mot}$ is weakly equivalent to $HR$. 
\end{lem}

\begin{proof}
    When $R=\Z$, this is \cite{Lev:ComparisonMotivicAndClassical}*{Prop. 5.6}. Otherwise, the ring $R$ admits a resolution by free abelian groups of the form
    \[
        0 \to \bigoplus_I \Z \to \bigoplus_J \Z \to R \to 0.
    \]
    Application of the functor $H(-)_{\mot}$ yields a distinguished triangle
    \[
        \bigvee_I H\Z_{\mot} \to \bigvee_J H\Z_{\mot} \to HR_{\mot} \xrar{\Sigma^{2,1}}
    \]
    in $\cat{SH}(k)$. Then $SRe_\C(HR_{\mot}) \simeq HR$ as $SRe_\C$ is triangulated.
\end{proof}

For each object $Y$ of $\cat{Spc}(k)$, the stable realization map $SRe_\C$ induces homomorphism
\begin{align*}
    \h^{p,q}(Y,R) = [Y_+,\Sigma^{p,q} HR_{\mot}]_{\cat{SH}(k)} \xrar{SRe_\C} [Re_\C(Y)_+,\Sigma^p HR]_{\cat{SH}} = \h^p(Re_\C(Y),R)
\end{align*}
which we denote by $Re_{\C,Y}$. The map $Re_{\C,Y}$ is natural in $Y$. If $(Y,y)$ is a pointed motivic space, there is also a reduced version which we denote $\tilde{Re}_{\C,Y}$. 
\begin{rem}\label{rem:ReRingHom}
    Since $SRe_\C$ is symmetric monoidal (\cite{PPR:VoevodskysKTheorySpectrum}*{Thm. A.45}), the map
    \[
        Re_{\C,Y}\co  \h^{*,*}(Y,R) \to \h^*(Re_\C(Y),R)
    \]
    is a graded ring homomorphism, where the $i$th graded piece of $\h^{*,*}(Y,R)$ is $\h^{i,*}(Y,R)$.
\end{rem}

The next lemma guarantees that the $\m_R$-algebra generators $\rho_i \in \h^{2i-1,i}(\GL_n,R)$ behave as expected under complex realization. 

\begin{lem}\label{lem:ReRho}
    There is a presentation
    \[
        \h^*(\U(n),R) = \Lambda_R(\rho_1^u, \dots, \rho_n^u)
    \]
    such that $Re_{\C,\GL_n}(\rho_i) = \rho_i^u$ for each $i$. 
\end{lem}
\begin{proof}
    It suffices to establish the result for $R=\Z$. The element $\rho_i \in \h^{2i-1,i}(\GL_n)$ is inductively defined as the unique element mapping to $\rho_i \in \h^{2i-1,i}(\GL_i)$ under the stabilization map $\GL_i \to \GL_n$. And $\rho_i \in \h^{2i-1,i}(\GL_i)$ is the image of a preferred generator in $\h^{2i-1,i}(\A^i \setminus 0)$ under the map induced by projection to the last column $\GL_i \to \A^i \setminus 0$. The element $\rho_i^u \in \h^{2i-1}(\U(n))$ is defined (up to a sign) similarly. It therefore suffices to show that 
    \[
        \tilde{Re}_{\C,\A^i \setminus \overline{0}}\co  \tilde{\h}{\vphantom{\h}}^{2i-1,i}(\A^i \setminus 0) \to \tilde{\h}{\vphantom{\h}}^{2i-1}(S^{2i-1})
    \]
    is an isomorphism, where $\A^i \setminus 0$ is pointed at the closed point $(1,0, \dots, 0)$. The homomorphisms $\tilde{Re}_{\C,-}$ are compatible with simplicial suspension, so it suffices to show
    \[
        \tilde{Re}_{\C,\p^i/\p^{i-1}}\co  \tilde{\h}{\vphantom{\h}}^{2i,i}(\p^i/\p^{i-1}) \to \tilde{\h}{\vphantom{\h}}^{2i}(\CP^i/\CP^{i-1})
    \]
    is an isomorphism. In this bidegree, the reduced and unreduced cohomology groups agree. Using the long exact sequence of a pair and the fact that $Re_{\C,\p^i}$ is a graded ring homomorphism, we reduce to showing
    \[
        Re_{\C,\p^i}\co  \h^{2,1}(\p^i) \to \h^2(\CP^i)
    \]
    is an isomorphism. This follows from the fact that $\p^\infty$ is a model for the motivic Eilenberg-MacLane space $K(\Z,2,1)$, and the natural map $\p^i \to \p^\infty$ corresponds to a generator of $\h^{2,1}(\p^i)$.
\end{proof}

\section{The splitting in characteristic zero}\label{sec:Char0Splitting}

In this section, we prove the splitting in characteristic zero. The strategy is first to establish the splitting over the prime field $\Q$ and then base change to prove the splitting over an arbitrary field of characteristic 0.

We say that two pure Tate motives $\bigoplus_i R(q_i)[p_i]$ and $\bigoplus_j R(q_j)[p_j]$ \textit{have the same Tate summands in Chow height $m$} if there is an equality $\{(p_i,q_i) \mid 2q_i-p_i=m\} = \{(p_j,q_j) \mid 2q_j-p_j = m\}$ of subsets of $\Z^2$ counted with multiplicities. 

\begin{lem}\label{lem:Bijection}
    For $0 \leq m < n_1$, the pure Tate motives $M(A(m,n_1, \dots, n_r))$ and $M(A(n_1, \dots, n_r))$ have the same Tate summands in Chow height $m$. 
\end{lem}
\begin{proof}
    \Cref{prop:APureTate} and \Cref{lem:FlPureTate} provide isomorphisms
    \begin{align*}
        M(A(m,n_1, \dots, n_r)) &\cong M(\GL_m) \otimes M(\Fl(m,n_1, \dots, n_r)) \\
        &\cong M(\GL_m) \otimes M(\Gr(m,n_1)) \otimes M(\Gr(n_1, n_2)) \otimes \cdots \otimes M(\Gr(n_{r-1},n_r)), \\
        M(A(n_1, \dots, n_r)) &\cong M(\GL_{n_1}) \otimes M(\Fl(n_1, \dots, n_r)) \\
        &\cong M(\GL_{n_1}) \otimes M(\Gr(n_1, n_2)) \otimes \cdots \otimes M(\Gr(n_{r-1},n_r)).
    \end{align*}
    Since the Tate summands of $M(\Gr(n_i,n_{i+1}))$ are concentrated in Chow height 0, we wish to show that $M(\GL_m) \otimes M(\Gr(m,n_1))$ and $M(\GL_{n_1})$ have the same Tate summands in Chow height $m$. The Chow height $m$ summands of the former are exactly 
    \[
        R(d(1,2, \dots, m)) \otimes \bigoplus_{\bm{\lambda} \in \Lambda} R(N(\bm{\lambda}))[2N(\bm{\lambda})]
    \]
    where $\Lambda$ is the set of $m$-tuples $\bm{\lambda} = (\lambda_1, \dots, \lambda_m)$ such that $\lambda_1 \leq \cdots \leq \lambda_m \leq n_1-m$, and $N(\bm{\lambda}) = \sum \lambda_i$. On the other hand, in light of \Cref{prop:GLnPureTate}, the Chow height $m$ summands of $M(\GL_{n_1})$ are given by
    \[
        \smashoperator[r]{\bigoplus_{1 \leq i_1 < \cdots < i_m \leq n_1}} R(d(i_1, \dots, i_m)).
    \]
    The Tate summand $R(d(1,2, \dots, m)) \otimes R(N(\bm{\lambda}))[2N(\bm{\lambda})]$ of the former corresponds to the Tate summand $R(d(\lambda_1+1,\lambda_2+2, \dots, \lambda_m+m))$ of $M(\GL_{n_1})$. 
\end{proof}

Given a pointed motivic space $(Y,y)$, let $M(Y,y)$ denote the reduced motive of $(Y,y)$: the cone on $y_*\co R \to M(Y)$. We point the scheme $A(n_1, \dots, n_r)$ at the $k$-rational point given by the isomorphism class of the diagram
\[
    \begin{tikzcd}
        k^{n_r} \rar & k^{n_{r-1}} \rar \lar[bend right] & \cdots \rar \lar[bend right] & k^{n_2} \rar \lar[bend right] & k^{n_1} \lar[bend right] \ar[loop right, looseness = 5,"\id"]
    \end{tikzcd}
\]
where each map $k^{n_{i+1}} \to k^{n_i}$ is the projection onto the first $n_i$ factors, and the splitting map $k^{n_i} \to k^{n_{i+1}}$ is the inclusion of the first $n_i$ factors. When $m \geq 1$, the motivic spaces $X_m(n_1, \dots, n_r)$ are canonically pointed. We abuse notation and denote the basepoint of any $X_m(n_1, \dots, n_r)$ by $*$.

\begin{prop}\label{prop:Mainprop}
    Suppose $k$ admits a complex embedding and that $(k,R)$ satisfies Beilinson-Soul\'{e} vanishing. For any sequence $n_1 < \cdots < n_r$ of nonnegative integers and any nonnegative $m \leq n_1$, the reduced motive $M(X_m(n_1, \dots, n_r),*)$ admits a decomposition
    \[
        M(X_m(n_1, \dots, n_r),*) \cong \smashoperator[lr]{\bigoplus_{\substack{1 \leq i_1 < \cdots < i_l \leq n_1, \\ m \leq l \leq n_1}}} M(\Fl(n_1, \dots, n_r))(d(i_1, \dots, i_l))
    \]
    such that the distinguished triangle
    \[
        \begin{tikzcd}
            M(X_{m}(m,n_1, \dots, n_r), *) \rar["{f_m^1(m,n_1, \dots, n_r)_*}"] &[3em] M(X_{m}(n_1, \dots, n_r), *) \rar &[-1.5em] M(X_{m+1}(n_1, \dots, n), *) \rar["{[1]}"] &[-1.5em] \;
        \end{tikzcd}
    \]
    is split in $\cat{DM}(k,R)$ and the map $f_m^1(m,n_1, \dots, n_r)_*$ is the inclusion of Tate summands of Chow height $m$. 
\end{prop}

The proof of \Cref{prop:Mainprop} is by nested induction, first on the integer $n_r$ and then on the integer $m$. 
    
\textbf{Base case $n_r=1$.} The only case to consider here is $r=m=1$. We have $X_0(0,1) = A(0,1) = \Spec(k)$ and $X_0(1) = A(1) = \GL_1$. Then $M(X_0(0,1), *) = 0$, and $M(X_0(1), *) = M(X_1(1), *) = \Z(d(1))$.

\textbf{Base case $m=0$ with $n_r > 1$.} Consider the $\A^1$-homotopy cofibre sequence 
\[
    \begin{tikzcd}
        \Fl'(n_1, \dots, n_r) \rar["{f^1(0,n_1, \dots, n_r)}"] &[2.5em] A(n_1, \dots, n_r) \rar &[-0.5em] X_1(n_1, \dots, n_r)
    \end{tikzcd}
\]
which induces the distinguished triangle of reduced motives
\[
    \begin{tikzcd}
        M(\Fl'(n_1, \dots, n_r),*) \rar["{f^1(0,n_1, \dots, n_r)}_*"] &[2.5em] M(A(n_1, \dots, n_r),*) \rar &[-0.5em] M(X_1(n_1, \dots, n_r),*) \rar["{[1]}"] &[-0.5em] \;
    \end{tikzcd}
\]
in $\cat{DM}(k,R)$. This triangle is split by the structure map $p\co  A(n_1, \dots, n_r) \to \Fl'(n_1, \dots, n_r)$ of the bundle of automorphisms (see \Cref{rem:AutBundle}). By \Cref{lem:IncTateSum}, we may arrange ${f^1(0,n_1, \dots, n_r)}_*$ to be the inclusion of those Tate summands of $M(A(n_1, \dots, n_r),*)$ with Chow height $0$.

\textbf{Inductive step: $r > 1$}. For this inductive step, suppose that $X_m(m,n_1, \dots, n_r)$ and $X_m(n_1, \dots, n_r)$ have the claimed Tate sum decompositions. We wish to show that
\[
    f_{m}^1(m,n_1, \dots, n_r)_*\co  M(X_{m}(m,n_1, \dots, n_r),*) \to M(X_{m}(n_1, \dots, n_r),*) \to X_{m+1}(n_1, \dots, n_r)
\]
is, up to isomorphism, the inclusion of Tate summands in Chow height $m$. The claimed decomposition of $M(X_{m+1}(n_1, \dots, n_r),*)$ follows readily from this.

First, we show that the source and target of $f_m^1(m,n_1, \dots, n_r)_*$ have the same Tate summands in Chow height $m$. Let $P_m$ denote the composition
\[
    A(n_1, \dots, n_r) = X_0(n_1, \dots, n_r) \to X_1(n_1, \dots, n_r) \to \cdots \to X_m(n_1, \dots, n_r),
\]
where each map is the second map in the $\A^1$-homotopy cofibre sequences
\[
    X_i(i,n_1, \dots, n_r) \xrar{f_i^1(i,n_1, \dots, n_r)} X_i(n_1, \dots, n_r) \to X_{i+1}(n_1, \dots, n_r).
\]
By the inductive hypothesis, the induced map
\[
    {P_{m}}_*\co  M(A(n_1, \dots, n_r),*) \to M(X_m(n_1, \dots, n_r),*)
\]
is the projection onto the Tate summands of Chow height at least $m$. In particular, the source and target of ${P_m}_*$ have the same Tate summands in Chow height $m$. The same is true for the corresponding map 
\[
    {P_{m}}_*\co  M(A(m,n_1, \dots, n_r),*) \to M(X_m(m, n_1, \dots, n_r),*).
\]
The claimed bijection of Chow height $m$ summands now follows from \Cref{lem:Bijection}. 

Since the diagram
\[
    \begin{tikzcd}
        M(A(m, n_1, \dots, n_r),*) \dar["{f^1(m,n_1, \dots,n_r)}_*"] \rar["{P_m}_*"] & M(X_m(m, n_1, \dots, n_r),*) \dar["{f_m^1(m,n_1, \dots,n_r)}_*"] \\
        M(A(n_1, \dots, n_r),*) \rar["{P_m}_*"] & M(X_m(n_1, \dots, n_r),*)
    \end{tikzcd}
\]
commutes, by Lemmas \ref{lem:CohoImpliesSplit} and \ref{lem:IncTateSum} it suffices to show that there are presentations of the source and target of 
\[
    {f^1(m,n_1, \dots, n_r)}^* \co  \h^{*,*}(A(n_1, \dots, n_r)) \to \h^{*,*}(A(m, n_1, \dots, n_r)) 
\]
as free bigraded $\m_R$-modules such that ${f^1(m,n_1, \dots, n_r)}^*$ is a bijection on the respective sets of module generators with Chow height $m$. To this end, we compare Rothenberg--Steenrod spectral sequences. As in \Cref{sec:SSCalcs}, let $\{\E_s^{*,*,*}, d_s\}$ denote the spectral sequence associated with the $\GL_{n_{r-1}}$-action on $V'(n_{r-1},n_r) \times A(n_1, \dots, n_{r-1})$ which converges to $\h^{*,*}(A(n_1, \dots, n_r))$. Let $\{\tensor*[^\prime]{\E}{_s^{*,*,*}}, \tensor*[^\prime]{d}{_s}\}$ denote the spectral sequence associated with the $\GL_{n_{r-1}}$-action on $V'(n_{r-1},n_r) \times A(m, n_1, \dots, n_{r-1})$ which converges to $\h^{*,*}(A(m,n_1, \dots, n_r))$. The map $f^1(m,n_1, \dots, n_r)$ is induced by the $\GL_{n_{r-1}}$-equivariant map
\[
    \id \times f^1(m,n_1, \dots, n_{r-1})\co  V'(n_{r-1},n_r) \times A(m,n_1, \dots, n_{r-1}) \to V'(n_{r-1},n_r) \times A(n_1, \dots, n_{r-1}),
\]
so we have an associated map of spectral sequences $\{\E_s^{*,*,*}, d_s\} \to \{ \tensor*[^\prime]{\E}{_s^{*,*,*}}, \tensor*[^\prime]{d}{_s}\}$.

By hypothesis, there are free $\m_R$-module presentations of the source and target of  
\[
    f^1(m,n_1, \dots, n_{r-1})^*\co  \h^{*,*}(A(n_1, \dots, n_{r-1}),R) \to \h^{*,*}(A(m, n_1, \dots, n_{r-1}),R)
\]
such that $f^1(m,n_1, \dots, n_{r-1})^*$ gives a bijection on $\m_R$-module generators in Chow height $m$. The associated map of $\E_2$-pages thus sends $\alpha_i'$ to $\alpha_i'$, $\theta_i$ to $\theta_i$, and induces a bijection on those elements $\beta_j'$ in total Chow height $m$. It follows from the \Cref{prop:SSDescription} that, on the $\E_\infty$-page, the $\m_R$-module generators in total Chow height $m$ in $\E_\infty^{*,*,*}$ map bijectively to the $\m_R$-module generators in total Chow height $m$ in $\tensor*[^\prime]{\E}{_\infty^{*,*,*}}$. As the $\m_R$-module generators in total Chow height $m$ of the respective spectral sequences give rise to $\m_R$-module generators in Chow height $m$ of $\h^{*,*}(A(m,n_1, \dots, n_r),R)$ and $\h^{*,*}(A(n_1, \dots, n_r),R)$, we are done.

\textbf{Inductive step: $r=1$.} Suppose that $M(X_m(m,n),*)$ and $M(X_m(n),*)$ have the claimed Tate sum decompositions. Since the Tate summands of $M(X_m(m,n),*)$ are concentrated in Chow height $m$ and the motives $M(X_m(m,n),*)$ and $M(X_m(n),*)$ have the same Tate summands in Chow height $m$ by hypothesis, it suffices to show that 
\begin{equation}\label{eq:f(m,n)*}
    f_m^1(m,n)^*\co  \tilde{\h}{\vphantom{\h}}^{p,q}(X_m(n),R) \to \tilde{\h}{\vphantom{\h}}^{p,q}(X_m(m,n),R)
\end{equation}
is an isomorphism whenever $2q-p = m$. The diagram 
\begin{equation}\label{eq:KeyDiagram}
    \begin{tikzcd}
        \tilde{\h}{\vphantom{\h}}^{*,*}(X_m(m,n),R) \dar["{\tilde{Re}_{\C,X_m(m,n)}}"] & \tilde{\h}{\vphantom{\h}}^{*,*}(X_m(n),R) \lar["{f_m^1(m,n)^*}"'] \dar["{\tilde{Re}_{\C,X_m(n)}}"]  \\
        \tilde{\h}{\vphantom{\h}}^*(F_m(n)/F_{m-1}(n),R) & \tilde{\h}{\vphantom{\h}}^*(\U(n)/F_{m-1}(n),R) \lar["i_m^*"']
    \end{tikzcd}
\end{equation}
commutes by \Cref{prop:XmRealization}, where $i_m\co F_m(n)/F_{m-1}(n) \to \U(n)/F_{m-1}(n)$ is the inclusion of spaces. Fix a bidegree $(p,q)$ with $2q-p = m$. By hypothesis, the groups $\tilde{\h}{\vphantom{\h}}^{p,q}(X_m(n),R),$ $\tilde{\h}{\vphantom{\h}}^{p,q}(X_m(m,n),R)$, and $\tilde{\h}{\vphantom{\h}}^p(F_m(n)/F_{m-1}(n),R)$ are free $R$-modules of the same rank. It therefore suffices to show that $Re_{\C,X_m(m,n)} \circ f_m^1(m,n)^* = i_m^* \circ Re_{\C,X_m(n)}$, as a map
\begin{equation}\label{eq:KeySurj}
    \tilde{\h}{\vphantom{\h}}^{p,q}(X_m(n),R) \to \tilde{\h}{\vphantom{\h}}^p(F_m(n)/F_{m-1}(n),R),
\end{equation}
is surjective. 

Fix a presentation $\h^*(\U(n),R) = \Lambda_R(\rho_1^u, \dots, \rho_n^u)$ such that $Re_{\C,\GL_n}(\rho_i) = \rho_i^u$ as in \Cref{lem:ReRho}. 

\begin{lem}\label{lem:Lev}
    There is a stable splitting of $\U(n)$ of the form \eqref{eq:MilSplitting} such that, in the induced direct sum decomposition in cohomology, the summand $\tilde{\h}{\vphantom{\h}}^*(F_m(n)/F_{m-1}(n),R)$ of $\h^*(\U(n),R)$ coincides with $\Lambda_R^m(\rho_1^u, \dots, \rho_n^u)$.
\end{lem}

\begin{proof}
    It suffices to prove the lemma for $R=\Z$. The stable splitting maps constructed in \cite{Crabb:StableSummands} establish the analogous statement for the Pontryagin ring $\h_*(\U(n),\Z)$. It follows from the universal coefficient theorem and the fact that the elements $\rho_i^u$ are primitive for the Hopf algebra $\h^*(\U(n),\Z)$ that there is a graded ring isomorphism $\h^*(\U(n),\Z) \cong \h_*(\U(n),\Z)$ given by taking the dual Hopf algebra. The result follows.
\end{proof}

Given a strictly increasing sequence $I = (i_1, \dots, i_l)$ in $\{1, \dots, n\}$, let $\rho_I^u$ denote the $l$-fold product $\rho_{i_1}^u \rho_{i_2}^u \cdots \rho_{i_l}^u$. Let $\Lambda_R^{\geq m} (\rho_1^u, \dots, \rho_n^u)$ denote the submodule
\[
    \bigoplus_{i \geq m} \Lambda_R^i (\rho_1^u, \dots, \rho_n^u) \subseteq \Lambda_R^*(\rho_1^u, \dots, \rho_n^u).
\]
Using \Cref{lem:Lev}, we may identify  $\tilde{\h}{\vphantom{\h}}^*(\U(n)/F_{m-1}(n),R)$ with the submodule $\Lambda_R^{\geq m}(\rho_1^u, \dots, \rho_n^u)$ of $\h^*(\U(n),R)$. Given a homogeneous cohomology class $\alpha$ in bidegree $(p,q)$, let $\ch(\alpha)$ denote the Chow height of $\alpha$: the integer $2q-p$. The following lemma establishes the surjectivity of \eqref{eq:KeySurj}, thus concluding the proof of \Cref{prop:Mainprop}.

\begin{lem}\label{lem:TrackClasses}
    Let $m \geq 0$, and assume that there are isomorphisms:
    \[
        M(X_m(m,n),*) \cong M(\Gr(m,n))(d(1,2, \dots, m)), \quad M(X_m(n),*) \cong \smashoperator[lr]{\bigoplus_{\substack{1 \leq i_1 < \cdots < i_l \leq n \\ m \leq l \leq n}}} R(d(i_1, \dots, i_l)).
    \]
    in $\cat{DM}(k,R)$. The following hold:
    \renewcommand{\labelenumi}{(\roman{enumi})}
    \begin{enumerate}
        \item For each strictly increasing sequence $I$ in $\{1, \dots, n\}$ with $|I| \geq m$, there is a homogeneous class $x_I \in \tilde{\h}{\vphantom{\h}}^{*,*}(X_m(n),R)$ with $\ch(x_I) = |I|$ and $\tilde{Re}(x_I) = \rho_I^u \in \tilde{\h}{\vphantom{\h}}^*(\U(n)/F_{m-1}(n),R)$.
        \item For any bidegree $(p,q)$ with $2q-p=m$, the map
        \[
            i_m^* \circ Re_{\C,X_m(n)}: \tilde{\h}{\vphantom{\h}}^{p,q}(X_m(n),R) \to \tilde{\h}{\vphantom{\h}}^p(F_m(n)/F_{m-1}(n),R)
        \]
        is surjective.
    \end{enumerate}
\end{lem}
\begin{proof}
    The proof is by induction on $m$. The base case $m=0$ is straightforward. Let $m > 0$, assume the result holds for $m-1$, and fix a sequence $I$ in $\{ 1, \dots, n \}$ with $|I| \geq m$. Realization of the cofibre sequence 
    \[
        \begin{tikzcd}[column sep = 4em]
            X_{m-1}(m-1,n) \rar["{f_{m-1}^1(m-1,n)}"] & X_{m-1}(n) \rar & X_m(n)
        \end{tikzcd}
    \]
    induces the commuting diagram 
    \[
        \begin{tikzcd}[column sep = small]
            0 \rar & \tilde{\h}{\vphantom{\h}}^{*,*}(X_m(n),R) \rar \dar["\tilde{Re}_{\C,X_m(n)}"] & \tilde{\h}{\vphantom{\h}}^{*,*}(X_{m-1}(n),R) \rar["{f_{m-1}^1(m-1,n)^*}"] \dar["\tilde{Re}_{\C,X_{m-1}(n)}"] &[2em] \tilde{\h}{\vphantom{\h}}^{*,*}(X_{m-1}(m-1,n),R) \rar \dar["\tilde{Re}_{\C,X_{m-1}(m-1,n)}"] & 0 \\
            0 \rar & \tilde{\h}{\vphantom{\h}}^*(\U(n)/F_{m-1}(n),R) \rar & \tilde{\h}{\vphantom{\h}}^*(\U(n)/F_{m-2}(n),R) \rar["i_{m-1}^*"]  &[2em] \tilde{\h}{\vphantom{\h}}^*(F_{m-1}(n)/F_{m-2}(n),R) \rar & 0
        \end{tikzcd}
    \]
    where the top row is (split) exact by hypothesis. There is a class $y_I \in \tilde{\h}{\vphantom{\h}}^{*,*}(X_{m-1}(n),R)$ in Chow height $|I|$ such that $\tilde{Re}_{\C,X_{m-1}(n)}(y_I) = \rho_I^u$ by the inductive hypothesis. Fix $\m_R$-module generators $\{\gamma_i\}_i$ of $\tilde{\h}{\vphantom{\h}}^{*,*}(X_{m-1}(m-1,n),R)$ with $\ch(\gamma_i) = m-1$ for each $i$. Then the set $\{Re_{\C,X_{m-1}(m-1,n)}(\gamma_i)\}_i$ forms a basis for $\tilde{\h}^*(F_{m-1}(n)/F_{m-2}(n),R)$, as $\tilde{\h}{\vphantom{\h}}^{p',q'}(X_{m-1}(n),R)$, $\tilde{\h}{\vphantom{\h}}^{p',q'}(X_{m-1}(m-1,n),R)$ and $\tilde{\h}{\vphantom{\h}}^{p'}(F_{m-1}(n)/F_{m-2}(n),R)$ are free $R$-modules of the same rank when $2q'-p'= m-1$ and the second part of the lemma holds for the integer $m-1$ by hypothesis. We have that $f_{m-1}^1(m-1,n)^*(y_I) = \sum_i c_i \gamma_i$ for some elements $c_i \in \m_R$. Then 
    \[
        \sum_i Re_{\C,\Spec k}(c_i) Re_{\C,X_{m-1}(m-1,n)}(\gamma_i) = Re_{\C,X_{m-1}(m-1,n)}(\sum_i c_i \gamma_i) = i_{m-1}^*(\rho_I^u) = 0
    \]
    by \Cref{lem:Lev}. Necessarily, $Re_{\C,\Spec k}(c_i) = 0$ for each $i$. We can find classes $\tilde{\gamma}_i \in \tilde{\h}{\vphantom{\h}}^{*,*}(X_{m-1}(n),R)$ in Chow height $m$ with $f_{m-1}^1(m-1,n)^*(\tilde{\gamma}_i) = \gamma_i$. Then the class $y_I - \sum_i c_i \tilde{\gamma}_i$ realizes to $\rho_I^u$ and is in the image of 
    \[
        \tilde{\h}{\vphantom{\h}}^{*,*}(X_m(n),R) \to \tilde{\h}{\vphantom{\h}}^{*,*}(X_{m-1}(n),R).
    \]
    This establishes the first part. With the help of \Cref{lem:Lev}, the second part follows from the first. 
\end{proof}

Given a motivic space $Y$ defined over the integers and a Noetherian scheme $S$ of finite Krull dimension, we let $Y_S$ denote the value of $Y$ under the base change functor $\cat{Spc}(\Z) \to \cat{Spc}(S)$. We remark that the motivic spaces $A(n_1, \dots, n_r)$ and $X_m(n_1, \dots, n_r)$ over $k$ are the base change of motivic spaces defined over $\Z$, as the flag varieties and the general linear groups are defined over $\Z$ and the functor $\cat{Spc}(\Z) \to \cat{Spc}(S)$ is left Quillen.

\begin{prop}\label{prop:Char0BC}
    Suppose $k$ is a characteristic $0$ field and $R$ is a commutative, unital ring. The conclusion of \Cref{prop:Mainprop} holds for the pair $(k,R)$.
\end{prop}
\begin{proof}
    The same arguments for the base cases in the proof of \Cref{prop:Mainprop} apply. For the inductive step (and for any $r$), we compare to the prime field $\Q$ for which the hypotheses of \Cref{prop:Mainprop} hold. Base change induces a diagram
    \[
        \begin{tikzcd}[column sep = 60]
            \tilde{\h}{\vphantom{\h}}^{p,q}(X_m(n_1, \dots, n_r)_\Q,R) \dar \rar["{f_m^1(m,n_1, \dots, n_r)^*}"] & \tilde{\h}{\vphantom{\h}}^{p,q}(X_m(m,n_1, \dots, n_r)_\Q,R) \dar \\
            \tilde{\h}{\vphantom{\h}}^{p,q}(X_m(n_1, \dots, n_r)_k,R) \rar["{f_m^1(m,n_1, \dots, n_r)^*}"] & \tilde{\h}{\vphantom{\h}}^{p,q}(X_m(m,n_1, \dots, n_r)_k,R)
        \end{tikzcd}
    \]
    We wish to show that the bottom horizontal map is an isomorphism when $2q-p=m$. The top horizontal map is an isomorphism by \Cref{prop:Mainprop}. The vertical maps are also isomorphisms since the motives involved are pure Tate (in $\cat{DM}(\Q,R)$ or $\cat{DM}(k,R)$, whatever the case may be) by hypothesis and 
    \[
        \h^{0,0}(\Spec \Q,R) \to \h^{0,0}(\Spec k,R)
    \]
    is an isomorphism. 
\end{proof}

The following is an important special case of \Cref{prop:Char0BC}. Note that $X_0(n) = X_1(n) = \GL_n$ by definition.

\begin{thm}\label{thm:MainThm}
    Suppose $k$ is a characteristic 0 field. For each $1 \leq m \leq n$ the distinguished triangle
    \[
        \begin{tikzcd}[column sep = large]
            M(X_m(m,n),*) \rar["{f_m^1(m,n)}_*"] & M(X_m(n),*) \rar & M(X_{m+1}(n),*) \rar["{[1]}"] & \;
        \end{tikzcd}
    \]
    splits in $\cat{DM}(k,R)$, inducing a decomposition
    \[
        M(\GL_n,I_n) \cong \bigoplus_{m=1}^n M(X_m(m,n),*).
    \]
    Moreover, if $k$ admits a complex embedding, the map $f_m^1(m,n)\co X_m(m,n) \to X_m(n)$ complex realizes (up to homotopy) to the inclusion $F_m(n)/F_{m-1}(n) \to \U(n)/F_{m-1}(n)$.
\end{thm}

\section{The splitting in positive characteristic}\label{sec:CharPSplitting}

To establish \Cref{thm:MainThm} in positive characteristic, one must invert the characteristic of the coefficient ring $R$ in the base field $k$. We begin by proving an analog of \Cref{prop:Mainprop} over the algebraically closed field $\overline{\F}_p$, then use base change arguments similar to the previous section to establish the result for any field $k$ of characteristic $p$. 

\begin{prop}\label{prop:FpBarSplitting}
    Let $p$ and $\ell$ be distinct primes. The conclusion of \Cref{prop:Mainprop} holds for $(k,R) = (\overline{\F}_p,\Z/\ell)$. 
\end{prop}
\begin{proof}
    In the proof of \Cref{prop:Mainprop}, the same arguments for the base cases apply, and since Beilinson--Soul\'{e} vanishing holds for the pair $(\overline{\F}_p, \Z/\ell)$, the spectral sequence arguments for the inductive step with $r>1$ apply. We just need to prove the inductive step with $r=1$, which, as in the characteristic-zero case, amounts to showing that the map in cohomology  
    \[
        f_m^1(m,n)^*\co \tilde{\h}{\vphantom{\h}}^{s,t}(X_m(n)_{\overline{\F}_p},\Z/\ell) \to \tilde{\h}{\vphantom{\h}}^{s,t}(X_m(m,n)_{\overline{\F}_p},\Z/\ell)
    \]
    is an isomorphism whenever $2t-s = m$, both source and target being free $\Z/\ell$-modules of the same rank by the inductive hypothesis. We show that $f_m^1(m,n)^*$ is an isomorphism in unreduced motivic cohomology.
    
    The continuous map from the \'{e}tale to the Nisnevich site gives a comparison
    \begin{equation}\label{eq:MapOfSites}
        \begin{tikzcd}[column sep = large]
            \h^{s,t}(X_m(n)_{\overline{\F}_p},\Z/\ell) \rar["{f_m^1(m,n)^*}"] \dar & \h^{s,t}(X_m(m,n)_{\overline{\F}_p},\Z/\ell) \dar \\
            \h_L^{s,t}(X_m(n)_{\overline{\F}_p}, \Z/\ell) \rar["{f_m^1(m,n)^*}"]& \h_L^{s,t}(X_m(m,n)_{\overline{\F}_p},\Z/\ell) 
        \end{tikzcd}
    \end{equation}
    where $\h_L^{s,t}(-,\Z/\ell)$ denotes Lichtenbaum (or \'{e}tale) motivic cohomology with $\Z/\ell$-coefficients. Following the techniques of \cite{Ray:ModulesProj}, there is a strictly Henselian local ring $\Lambda$ with residue field $\overline{\F}_p$ and fraction field $K(\Lambda)$ that admits a complex embedding (e.g., $\Lambda = W(\overline{\F}_p)$, the ring of Witt vectors of $\overline{\F}_p$). There are base change functors
    \begin{equation}\label{eq:DMBaseChange}
        \cat{DM}_{\et} (\overline{\F}_p, \Z/\ell) \leftarrow \cat{DM}_{\et} (\Lambda, \Z/\ell) \to \cat{DM}_{\et} (K(\Lambda), \Z/\ell) \to \cat{DM}_{\et} (\C,\Z/\ell).
    \end{equation}
    where $\cat{DM}_{\et}(-, \Z/\ell)$ denotes the fibred category of \'{e}tale motives with $\Z/\ell$-coefficients (see \cite{CD:EtaleMotives} for a precise definition). For a smooth $\Z$-scheme $Y$, there are natural isomorphisms
    \[
        \h_L^{s,t}(Y_S,\Z/\ell) \cong \h_{\et}^s(Y_S, \Z/\ell) 
    \]
    where $S$ is any of the base schemes in \eqref{eq:DMBaseChange} and the left hand side is computed in $\cat{DM}_{\et}(S,\Z/\ell)$ (even $S = \Lambda$, see \cite{CD:EtaleMotives}*{Thm. 4.5.2}). The homomorphisms in Lichtenbaum cohomology induced by \eqref{eq:DMBaseChange} are isomorphisms for smooth $\Z$-schemes by \cite{Ray:ModulesProj}*{Thm. 5.2}. Recall from \Cref{sec:AuxSpaces} that the motivic spaces $X_m(n_1, \dots, n_r)$ are iterated homotopy cofibres, and that $X_1(n_1, \dots, n_r)$ is the homotopy cofibre of a map of representables. An inductive argument, utilizing the fact that the functors in \eqref{eq:DMBaseChange} are triangulated, shows that \eqref{eq:DMBaseChange} induces a zig-zag of isomorphisms in Lichtenbaum cohomology of (base changes of) $X_m(n_1, \dots, n_r)$. In particular, there is a comparison
    \begin{equation}\label{eq:BCZigZag}
    \begin{tikzcd}[column sep = large]
        \h_L^{*,*}(X_m(n)_{\overline{\F}_p}, \Z/\ell) \dar["\cong"] \rar["{f_m^1(m,n)^*}"] & \h_L^{*,*}(X_m(m,n)_{\overline{\F}_p},\Z/\ell) \dar["\cong"] \\
        \h_L^{*,*}(X_m(n)_\C, \Z/\ell) \rar["{f_m^1(m,n)^*}"] & \h_L^{*,*}(X_m(m,n)_\C,\Z/\ell) 
    \end{tikzcd}
    \end{equation}
    induced by a zig-zag of base changes. The \'{e}tale motivic Eilenberg-MacLane spectrum $H\Z/\ell_{\et}$ in $\cat{SH}(\C)$ complex realizes to $H \Z/\ell$ since any choice of Bott element, thought of as a morphism $\beta\co H \Z/\ell_{\mot} \to \Sigma^{0,1} H \Z/\ell_{\mot}$, complex realizes to a weak equivalence $H\Z/\ell \to H\Z/\ell$ in $\cat{SH}$. Complex realization thus induces a graded ring homomorphism
    \[
        Re_{\C,Y}\co \h_L^{*,*}(Y_\C,\Z/\ell) \to \h^*(Re_{\C}(Y), \Z/\ell)
    \]
    which, for a fixed bidegree and when $Y$ is a smooth scheme, coincides with the comparison isomorphism of \cite{Milne}*{Thm. III.3.12}. Put together, we have comparisons 
    \begin{equation}\label{eq:FpComp}
        \begin{tikzcd}[column sep = large]
            \h^{s,t}(X_m(n)_{\overline{\F}_p},\Z/\ell) \rar["{f_m^1(m,n)^*}"] \dar["\phi_{X_m(n)}"] & \h^{s,t}(X_m(m,n)_{\overline{\F}_p},\Z/\ell) \dar["\phi_{X_m(m,n)}"] &  \\
            \h_L^{s,t}(X_m(n)_\C, \Z/\ell) \rar["{f_m^1(m,n)^*}"] \dar["\cong"',"Re_{\C,X_m(n)}"] & \h_L^{s,t}(X_m(m,n)_\C,\Z/\ell) \dar["\cong"',"Re_{\C,X_m(m,n)}"] \\
            \h^s(\U(n)/F_{m-1}(n),\Z/\ell) \rar["i_m^*"] & \h^s(F_m(n)/F_{m-1}(n),\Z/\ell) 
        \end{tikzcd}
    \end{equation}
    where the lower vertical maps are induced by complex realization, and $\phi_{(-)}$ is the composition of \eqref{eq:MapOfSites} and \eqref{eq:BCZigZag}. The two realization maps in \eqref{eq:FpComp} are isomorphisms since the motivic spaces $X_m(n)_\C$ and $X_m(m,n)_\C$ are iterated homotopy cofibres of representables and complex realization is left Quillen. 
    
    We now proceed with a diagram chase similar to the proof of \Cref{prop:Mainprop}. Since $\h^{s,t}(X_m(n)_{\bar{\F}_p}, \Z/\ell)$, $\h^{s,t}(X_m(m,n)_{\bar{\F}_p}, \Z/\ell)$, $\h_L^{s,t}(X_m(m,n)_\C, \Z/\ell)$, and $\h^s(F_m(n)/F_{m-1}(n), \Z/\ell)$ are free $\Z/\ell$-modules of the same rank, it suffices to show that the composite
    \[
        Re_{\C,X_m(m,n)} \circ \phi_{X_m(m,n)} \circ f_m^1(m,n)^*: \h^{s,t}(X_m(n)_{\bar{\F}_p}, \Z / \ell) \to \h^s(F_m(n) / F_{m-1}(n),\Z / \ell)
    \]
    in \eqref{eq:FpComp} is surjective. An inductive argument similar to that of \Cref{lem:TrackClasses} shows that for each sequence $I$ in $\{1, \dots, n \}$ with $|I| \geq m$, there is a class $x_I \in \h^{*,*}(X_m(n)_{\bar{\F}_p}, \Z/ \ell)$ with $\ch(x_I) = |I|$ and $Re_{\C,X_m(n)} \circ \phi_{X_m(n)}(x_I) = \rho_I$. With the help of \Cref{lem:Lev}, this proves the claim.
\end{proof}

\begin{prop}
    Let $p$ and $\ell$ be distinct primes, and suppose $k$ is a field of characteristic $p$. The conclusion of \Cref{prop:Mainprop} holds for the pair $(k,\Z/\ell)$.
\end{prop}
\begin{proof}
    Again, the base cases are easy to establish. For the inductive step, the base change induced by $\F_p \to \overline{\F}_p$ and appeal to \Cref{prop:FpBarSplitting} (\textit{c.f.}~the proof of \Cref{prop:Char0BC}) establishes the conclusion of \Cref{prop:Mainprop} for $(k,R) = (\F_p,\Z/\ell)$. Comparison to the prime field $\F_p$ establishes the result for $(k,\Z/\ell)$ where $k$ is any characteristic $p$ field. 
\end{proof}

\appendix

\section{Distinguished triangles and pure Tate motives}

We collect two basic results regarding distinguished triangles in $\cat{DM}(k,R)$ in which one or more of the objects in the triangle is pure Tate.

\begin{lem}\label{lem:CohoImpliesSplit}
    Suppose 
    \[
        M \xrar{f} N \to C(f) \xrar{[1]}
    \]
    is a distinguished triangle in $\cat{DM}(k,R)$ and $M \cong \bigoplus_{i \in I} R(q_i)[p_i]$ is pure Tate. Then the distinguished triangle splits if the induced map 
    \[
        f^*: \bigoplus_{i \in I} \h^{p_i,q_i}(N,R) \to \bigoplus_{i \in I} \h^{p_i,q_i}(M,R)
    \]
    is surjective. 
\end{lem}

The following is a straightforward exercise in algebra. The key observation is that a map $f\co M \to N$ of pure Tate objects in $\cat{DM}(k,R)$ (with specified Tate sum decompositions of $M$ and $N$) is the same as a matrix with entries in the motivic cohomology groups of $\Spec k$. 

\begin{lem}\label{lem:IncTateSum}
    Suppose 
    \[
        M \xrar{f} N \to C(f) \xrar{[1]}
    \]
    is a split distinguished triangle in $\cat{DM}(k,R)$, that $M$ and $N$ are pure Tate, and that the Tate summands of $M$ are concentrated in Chow height $m$. Suppose further that $M$ and $N$ have the same Chow height $m$ Tate summands. Then, up to isomorphism, the map $f\co M \to N$ is the inclusion of Chow height $m$ Tate summands. In particular, $C(f)$ is pure Tate, and the Tate summands of $C(f)$ correspond to those Tate summands of $N$ with Chow height different from $m$. 
\end{lem}

\clearpage

\bibliographystyle{alpha}
\bibliography{References}

\end{document}